\documentclass[11pt]{amsart}
\usepackage[utf8]{inputenc}
\usepackage{amsmath, mathtools}
\usepackage{amsfonts, amsthm}
\usepackage[hmargin=3.5cm,vmargin=3.5cm]{geometry}
\usepackage{amssymb}
\usepackage{enumitem}
\usepackage{mathabx}
\usepackage{mathrsfs}  
\usepackage{mathtools}
\usepackage{setspace}\usepackage{bm}
\usepackage{color,soul}
\usepackage{tikz-cd}

\newcommand{\R}{\mathcal{R}}

\newcommand{\sS}{\mathcal{S}}

\newcommand{\sln}{\mathfrak{sl}(n)}
\newcommand{\Ur}{\mathrm{U}(r)}

\newcommand{\U}{\mathrm{U}}
\newcommand{\E}{\mathrm{E}}

\newcommand{\htor}{\mathrm{H}_{\mathrm{T}}}
\newcommand{\thetah}{\prescript{\theta}{}{\mathcal H}}
\newcommand{\nHH}{\prescript{n}{}{\mathrm{H}}}
\newcommand{\uHH}{\prescript{u}{}{\mathcal{H}}}
\newcommand{\KRHn}{\widehat{KRH}_n^\Z}
\newcommand{\KhRn}{\widehat{KhR}_n^\Z}
\newcommand{\KRHu}{\widehat{KRH}_u^\R}

\newcommand{\thetae}{\prescript{\theta}{}{E}}

\newcommand{\bet}{\boldsymbol\beta}
\newcommand{\coe}{\mathfrak{U}}
\newcommand{\fil}{\mathcal{F}}
\newcommand{\BB}{\mathcal{B}}
\newcommand{\sym}{\Z[x_1,\dots,x_k]^W}
\newcommand{\tT}{\mathrm{T}}

\newcommand{\twE}{\widehat{E}}
\newcommand{\HH}{\mathrm{H}}

\newcommand{\BTS}{\mathcal{B}t\mathcal{S}}
\newcommand{\Q}{\mathbb{Q}}
\newcommand{\C}{\mathbb{C}}
\newcommand{\unlink}{\ovoid}
\newcommand{\II}{\mathcal{I}}
\newcommand{\Z}{\mathbb{Z}}
\newcommand{\ctp}[1][]{\,\widehat\otimes_{#1}\,}

\newcommand{\EE}{\mathcal{E}}
\newcommand{\GL}{\mathrm{GL}_1}
\newcommand{\hoch}{\mathsf{HH}}
\newcommand{\gr}{\mathrm{gr}}

\newcommand{\coebs}{\mathfrak{F}}
\newcommand{\sgn}{\mathrm{sgn}}
 \newcommand{\Ch}{\mathsf{Ch}}
  \newcommand{\Br}{\mathrm{Br}}
 \newcommand{\ringD}{\R[\![X_e]\!]}
 \newcommand{\bup}{B\mathrm{U}_+}
  \newcommand{\hcal}{\mathcal{H}}
  \DeclareMathOperator{\assgr}{\mathsf{gr}}
  \DeclareMathOperator{\colim}{colim}
 
\newtheorem{theorem}{Theorem}

\newtheorem{corollary}{Corollary}
\newtheorem{proposition}{Proposition}
\newtheorem{definition}{Definition}
\newtheorem{lemma}{Lemma}

\theoremstyle{remark}
\newtheorem*{remark}{Remark}

\let\oldtocsection=\tocsection
\let\oldtocsubsection=\tocsubsection
\let\oldtocsubsubsection=\tocsubsubsection
\renewcommand{\tocsection}[2]{\hspace{0em}\oldtocsection{#1}{#2}}
\renewcommand{\tocsubsection}[2]{\hspace{1em}\oldtocsubsection{#1}{#2}}
\renewcommand{\tocsubsubsection}[2]{\hspace{2em}\oldtocsubsubsection{#1}{#2}}

\title{Khovanov-Rozansky homologies, Bott-Samelson spaces and twisted cohomology}
\author{Tom\'as Mej\'ia-G\'omez}
\address{Department of Mathematics, Johns Hopkins University, Baltimore, MD 21218, USA}
\email{mejia@jhu.edu}

\usepackage{hyperref}

\hypersetup{
    colorlinks,
    linkcolor={red!50!black},
    citecolor={blue!50!black},
    urlcolor={blue!80!black}
}

\begin{document}

\begin{abstract}
By means of Rasmussen's formulation of Khovanov-Rozansky homology originally given over $\Q$ in \cite{rasmussen}, we compare different flavors of $\sln$-link homology  with the link invariants obtained by Kitchloo in \cite{nitubroken2} via twistings of Borel equivariant cohomology applied to the symmetry breaking spectra. In particular, we see how these geometric constructions based on Bott-Samelson varieties produce equivariant integral $\sln$ link homology with either specialized or universal potential.
\end{abstract}

\maketitle

\tableofcontents

\setstretch{1.1}

\section*{Introduction}

In the papers \cite{krmatrix} and \cite{krmatrix2}, Khovanov and Rozansky introduced link homologies which categorify the classical link invariants known as the $\sln$ polynomial $P_n$ and the HOMFLY-PT polynomial $P_\infty$. These polynomials can be normalized in such a way that, for the unknot $\unlink$, one has
$$P_n(\unlink)(q)=\frac{1-q^{2n}}{1-q^2}\quad\mbox{and}\quad P_\infty(\unlink)(q,z)=\frac{1-z}{1-q^2}.$$ The corresponding Khovanov-Rozansky homologies give, in turn,
$$KRH_n(\unlink)=\frac{\Q[x]}{(x^n)} \quad\mbox{and}\quad KRH_\infty(\unlink)=\Q[x]\otimes\Lambda(\hat x).$$
With an adequate grading $\gr=(\gr_t,\gr_q,\gr_z)$ (ignoring $\gr_z$ for $H_n$) in which $\gr(x)=(2,2,0)$ and $\gr(\hat x)=(1,0,1)$, the link polynomial is recovered by taking Euler characteristic with respect to the degree $\gr_t$. A slight reformulation of the Khovanov-Rozansky complex that computes $KRH_n$ allowed Rasmussen \cite{rasmussen} to find a spectral sequence with $E_1$-page given by HOMFLY-PT homology, and converging to $\sln$ homology. This spectral sequence should be seen as a categorification of the relation $P_n(q)=P_\infty(q,q^{2n})$. In broad terms, HOMFLY-PT homology $KRH_\infty$ can be defined through a complex $C(D)$ obtained from a cube of resolutions of a link diagram $D$, with a \emph{cubical differential} $d_v$. The cubical differential can be `twisted' by adding another differential $d_-$ (depending on $n$) to obtain $\sln$ homology as the homology with respect to the total differential $d_v+d_-$. The Rasmussen spectral sequence is the standard double complex spectral sequence in this setting.

 Rasmussen's work was done over the field $\Q$, but Krasner \cite{krasnerequivariant} extended the coefficients to $\Q[a_0,\dots,a_{n-2}]$ in order to obtain equivariant $\sln$ homology, which takes the form $$KRH^{\Q[a_0,\dots,a_{n-2}]}_p(\unlink)=\frac{\Q[a_0,\dots,a_{n-2}][x]}{(p(x))}$$
at the unknot. Here $p(x)=x^n+a_2x^{n-2}+\cdots+a_1x+a_0$, a sort of degree $n$ monic polynomial with generic coefficients. Over $\Z$, Krasner describes an integral version of the Rasmussen spectral sequence in \cite{krasnerintegral}. By the lack of an a priori integral link invariant, one is led to define $\sln$ homology over $\Z$ as either the $E_\infty$-page or the limiting object of the spectral sequence. Either way, there is a link homology which at the unknot is given by
$$KRH^\Z_n(\unlink)=\frac{\Z[x]}{((n+1)x^n)}.$$
A normalized version of integral $\sln$ homology would remove the integral torsion. This has been achieved through various different routes in the literature. Using known geometric constructions, the present work proposes such a normalization while keeping with the spirit of Rasmussen's double complex.

In order to introduce the spaces of interest, it is convenient to work with a braid presentation for the link $L$. In the positive case, take a multi-index $I=(l_1,\dots,l_n)$ with $0<l_s<r$, and a positive braid word $w_I=\sigma_{l_1}\cdots \sigma_{l_s}$ representing an element of the $r$-strand braid group $\Br(r)$, so that the closure of the braid is a diagram $D$ representing the link $L$. On the algebra side, one associates to $w_I$ a \emph{Soergel bimodule} $B_I$ over the polynomial ring $\Z[x_1,\dots,x_r]$, with Hochschild homology $\hoch(B_I)$. The whole poset of subindices (subwords) $J\leq I$, which is combinatorially a cube, indexes a \emph{cubical complex} $C(D)=\bigoplus_J \hoch(B_J)$, with a cubical differential $d_v$ defined using elementary maps between Soergel bimodules. After rationalization, this is one possible description of the complex $(C(D),d_v)$ whose homology is $KRH_\infty(L)$. Notice that $n$ plays no role so far because the twisting differential $d_-$ has not yet been involved.  

Geometry comes into the picture following the observation that Soergel bimodules $B_J$ are torus equivariant cohomologies of the Bott-Samelson varieties $\BTS_\tT(w_J)$, and relevant maps between the bimodules are induced from natural maps between the varieties. These varieties are built out of various subgroups of the unitary group $\Ur$. Concretely, we take as building block $G_i\leq \Ur$, the block diagonal subgroup with a $2\times 2$ block in positions $i$ and $i+1$ and $1\times 1$ blocks elsewhere. The significant $\tT$-spaces are the Bott-Samelson variety
$$\BTS_\tT(w_J)= G_{i_1}\times_\tT\cdots\times_\tT G_{i_l}/\tT,$$
 for which its Borel equivariant cohomology is the Soergel bimodule $\htor^*(\BTS_\tT(w_J))\cong B_J$, and also the space
$$\BB_\tT(w_J)= G_{i_1}\times_\tT\cdots\times_\tT G_{i_l},$$ which is a principal $\tT$-bundle over $\BTS_\tT(w_J)$ and satisfies $\htor^*(\BB_\tT(w_J))\cong \hoch(B_J)$. The $\tT$-action on $\BTS_\tT(w_J)$ is by left multiplication on the first factor $G_{i_1}$, while the action on $\BB_\tT(w_J)$ is by conjugation on all factors. These spaces and their relationship to Hochschild homology of Soergel bimodules are explored in an algebro-geometric language in \cite{webwil}, and they are also the starting point for the stable homotopy link invariant in \cite{nitubroken1}. In the latter, Kitchloo glues suspensions of the various $\BB_\tT(w_J)$ for $J\leq I$ to define a \emph{filtered $\Ur$-spectrum of strict broken symmetries} $s\BB(L)$, and its equivariant stable homotopy type is shown to be a link invariant up to a notion of \emph{quasi-equivalence}. Given the relation to Soergel bimodules, it is not surprising that integral HOMFLY-PT homology is recovered through Borel equivariant cohomology $\htor^*$. 

 In this context of strict broken symmetries, the notion of deforming HOMFLY-PT into $\sln$ is achieved by twisting the cohomology theory $\HH^*$. A first approach, still algebraic in flavor, is to take advantage of the structure of $\htor^*(\BB_\tT(w_J))$ as an algebra over $\HH^*(\mathrm{U})$, the non-equivariant cohomology of the infinite unitary group $\U$. This structure comes from a composite map 
 $$E\tT\times_{\tT}\BB_\tT(w_J)\to E\tT\times_{\tT}\Ur\to \U$$ in which the first arrow is induced from a canonical $\tT$-equivariant multiplication map (with respect to conjugation action) $\BB_\tT(w_J)\to \Ur$, and the second arrow can be put in concrete geometric terms of frames and matrices. The ring $\HH^*(\mathrm{U})$ is an exterior algebra in generators $\beta_n$, $n=0,1,2,\dots$ with degree $|\beta_n|=2n+1$. For fixed $n$, the action of $\beta_n$ on $\htor^*(\BB_\tT(w_J))$, being an exterior generator, defines a differential $d_-$. We can take 
 $$\nHH_\tT^*(\BB_\tT(w_J))=H\left(\htor^*(\BB_\tT(w_J)),d_-\right),$$
and think of $\nHH_\tT^*$ as a homology functorial on $\tT$-spaces over $\Ur$. Assembling into a cube $\bigoplus_{J\leq I} \nHH_\tT^*(\BB_\tT(w_J))$ and taking homology with respect to the induced cubical differential gives an integral link homology with torsion at the unknot generated by $x^n$. Formulated concisely, here is what we set out to prove.

\begin{theorem}
Through application of the theory $\nHH_\tT^*$ to the $\tT$-spaces $\BB_\tT(w_J)$ one obtains an integral version of Rasmussen's double complex with differentials $d_-$ and $d_v$. Homologies of this complex compute an $\sln$ homology theory of links with value at the unknot given by 
$$\KRHn(\unlink)=\frac{\Z[x]}{\left(x^n\right)}.$$
Scaling the differential $d_-$ by $n+1$ recovers the un-normalized integral homology $KRH^\Z_n$.
\end{theorem}
 
This approach can also be used to obtain an integral equivariant link homology in the vein of Krasner's. For this, we work with the cohomology theory $\hcal^*(X)=\HH(X)\ctp \R$, where $\R=\Z[b_1,b_2,...]$ and $b_i$ has cohomological degree $-2i$.  The reader should be wary about the ambiguous terminology: topologically we work with $\tT$-spaces and their equivariant cohomology, but when we qualify link homologies as `equivariant' we are referring to extending coefficients by formal polynomial variables such as $b_i$. In this extended theory, the cohomology of $\mathrm U$ is given by $\hcal^*(\mathrm U)=\Lambda (\beta_0,\beta_1,\beta_2,\dots)\ctp\R$, and there is a universal class $\beta_u= \sum_{i>0} b_i\beta_i \in \hcal^1(\mathrm U)$, i.e. a combination of all the $\beta_n$ for $n>0$ with the formal parameters $b_i$ as coefficients. We can, again, define a differential $d_-$ via multiplication by $\beta_u$ and thus get
$$\uHH_\tT^*(\BB_\tT(w_J))=H\left(\hcal_\tT^*(\BB_\tT(w_J)),d_-\right),$$
which is a functor of $\tT$-spaces over $\Ur$. In this case, we obtain a link homology with torsion at the unknot generated by the generic (or \emph{universal}) power series $\sum_{i>0}^\infty b_ix^i$.

\begin{theorem}
Through application of the theory $\uHH_\tT^*$ to the $\tT$-spaces $\BB_\tT(w_J)$ one obtains a `universal integral equivariant' $\sln$ homology theory of links with value at the unknot given by 
$$\KRHu(\unlink)=\frac{\Z[b_1,b_2,\dots][\![x]\!]}{(\sum_{i>0} b_ix^i)}.$$
\end{theorem}

The double complex $\left(\bigoplus_{J\leq I}\hcal^*_\tT(\BB_\tT(w_J)),d_-,d_v\right)$ can be specialized to get other versions of $\sln$ link homology: set all $b_i=0$ except for $b_n=1$ to get integral homologies $KRH^\Z_n$, rationalize to obtain rational homologies $KRH_n$, or keep finitely many of the $b_i$ as variables to produce equivariant homologies similar to $KRH^{\Q[a_0,\dots,a_{n-2}]}_p$.

The functor $\uHH_\tT^*$ appears in \cite{nitubroken2} as an approximation to a twisted cohomology theory that we denote here by $\thetah^*_\tT$. To get a sense of what this twisted theory is, one begins with the observation that the cohomology theory $\hcal^*$ can be represented by the spectrum $\mathcal H = \HH\Z\wedge \bup$, where $\HH\Z$ is the integral Eilenberg-Maclane spectrum and $B\U$ is the classifying space of $\U$. The spectrum $\hcal$ admits a suitable ring structure, allowing us to apply the formalism of twisted cohomology theories (e.g. \cite{satiwester}, \cite{abghr}). Concretely, there is a twist $\theta:\U\to B\GL(\hcal)$, ultimately originating from Bott periodicity, giving a theory $\thetah^*$ for spaces over $\mathrm U$. It can be applied, in particular, to the Borel construction on $\BB_\tT(w_J)$, in which case we write
$$\thetah^*_\tT(\BB_\tT(w_J)):=\thetah^*(\E\tT\times_\tT\BB_\tT(w_J)).$$
The relation between $\thetah^*_\tT(\BB_\tT(w_J))$ and $\uHH^*_\tT(\BB_\tT(w_J))$ is a spectral sequence whose collapse is proved in our last section.
\begin{theorem}
There is a spectral sequence with signature
\begin{equation*}\thetae_1^{*,*}(X)=\hcal^*_\tT(\BB_\tT(w_J)) \Longrightarrow \thetah^{*}_\tT(\BB_\tT(w_J)),\end{equation*}
with differential $d_1$ given by multiplication by $\beta_u$, and thus $$\thetae_2^{*,*}(X)=\uHH^*_\tT(\BB_\tT(w_J)).$$
The spectral sequence degenerates at the second page, and in consequence $\uHH^*_\tT(\BB_\tT(w_J))$ is an associated graded of $\thetah^*_\tT(\BB_\tT(w_J))$.
\end{theorem}

On the nose, the spectrum $s\BB(L)$ seems to be a rather strong topological link invariant, capturing some well known algebraic invariants when observed through adequate cohomology theories. This approach to the study of link homology can serve a double purpose. On the one hand, the topological point of view on known link homologies such as HOMFLY-PT or $\sln$ could reveal further structure through standard algebraic topology machinery such as cohomological operations. On the other hand, theories like $\HH$ and $\thetah^*$, as well as other well developed cohomologies such as different flavors of K-theory, can be explored and compared to reveal new link invariants or provide a framework for studying various kinds of relations between link homologies. A concrete direction is suggested by results relating HOMFLY-PT and knot Floer homologies, see \cite{manolescu2014untwisted} and \cite{BPRW}. The latter work exploits interesting variations on Hochschild homology of Soergel bimodules that may admit topological interpretation by choosing the right (twisted) cohomology theory.

\noindent\textbf{General structure of the paper.}  In the first section we review the Soergel bimodule approach to HOMFLY-PT homology, and explain how it is obtained from the spaces $\BTS_\tT(w_J)$ and $\BB_\tT(w_J)$. We also review how negative crossings can be accommodated and what requirements should we ask of more general cohomologies to produce algebraic link invariants out of $s\BB(L)$. In the second section we present the original Khovanov-Rozansky complex.  We discuss how this complex can be reduced to the Soergel bimodule picture, and the Rasmussen spectral sequence is laid out there as well. The third section defines $\nHH_\tT^*$ and proves technical results about $\nHH_\tT^*(\BB_\tT(w_J))$ that allow for comparison to the algebraic setting from Section 2. In the last section we do a general overview of twisted theories and the particular twist $\theta$ on $\hcal=\HH\wedge\bup$. Then we elaborate on the relation between $\thetah^*_{\tT}$ and $\uHH^*_{\tT}$.

\noindent\textbf{Notation guide.} For the reader's convenience, we provide a quick reference notation table for various homologies and cohomologies used in this work.

{\renewcommand{\arraystretch}{1.5}
\begin{tabular}{p{1.8cm}| p{11.5cm}}

 $\HH^*$ & Singular cohomology with coefficients in $\Z$, represented by the Eilenberg Mac-Lane spectrum $\HH=\HH\Z$. \\$\hcal^*$ & Singular cohomology with coefficients in $\R=\Z[b_1,b_2,\dots]$, represented by the spectrum $\hcal=\HH\Z\wedge \bup$. \\ 
  $\HH_\tT^*, \hcal_\tT^*$ & Borel equivariant versions, with $\tT$-spaces as inputs.\\  
 $\nHH_\tT^*, \uHH_\tT^*$ & Algebraically twisted versions, obtained as homology with respect to multiplication by an odd class. Their inputs are $\tT$-spaces over $\Ur$. \\
  $\thetah_\tT^*$ & Topologically twisted version of $\hcal_\tT^*$, with twist defined as a map $\U\to B\GL (\hcal)$. Its inputs are $\tT$-spaces over $\Ur$.\\
  $H(-,d_\square)$ or $H_\square$ & Homology of a complex with respect to a differential $d_\square$, where $\square$ is some decoration.\\
  $\hoch$ & Hochschild homology of a bimodule over a polynomial ring.
\end{tabular}
}
 
\noindent\textbf{Acknowledgments.} The author would like to thank his advisor Nitu Kitchloo for suggesting this project and all the continuous support. He would also like to thank Anish Chedalavada and Mikhail Khovanov for helpful discussions, and a special thank you to Valentina Zapata Castro for her overall support during the time working on this project.

\section{Ordinary Borel cohomology and triply graded link homology}

\newcommand{\topur}{\mathsf{Top}_{\Ur}}
\newcommand{\term}{J_{\mathrm{\infty}}}

As stated in the introduction, we are interested in determining what algebraic invariants of links can be extracted from them spaces $\BB_\tT(w_J)$ by applying cohomology theories to them. The first instance is HOMFLY-PT homology, obtained via Borel equivariant cohomology. This section is devoted to sketching the relevant aspects of HOMFLY-PT homology and $\BB_\tT(w_J)$, and then having a first glance at how the machinery of spectral sequences helps establish the relation between them.

\subsection{HOMFLY-PT homology}

We assume the link $L$ has a presentation as the closure of a braid word $w=w_I=\sigma_{l_1}\cdots\sigma_{l_n}$ in the classical braid group on $r$ strands, $\mathrm{Br}(r)$, with standard generators $\sigma_1,\dots,\sigma_{r-1}$. For convenience of notation, we will write $\sigma_{-l}=\sigma_l^{-1}$, so subindices are allowed to take values $l=\pm 1,\dots,\pm (r-1)$. Denote by $\II$ the poset consisting of all functions $J:\{1,\dots,n\}\to\{0,1\}$, where $J_1\leq J_2$ if $\sgn(l_j)J_1(j)\leq \sgn(l_j)J_2(j)$ for all $j$. More often we will think of $J$ as the subsequence $J=(l_{j_1},\dots,l_{j_k})\subseteq I$, where $\{j_1,\dots,j_k\}=J^{-1}(1)$. In such situations we clearly refer to $J$ as a (sub)sequence and we may write $J=(i_1,\dots,i_k)=(l_{j_1},\dots,l_{j_k})$ to avoid too many levels of subindices. As categories, $\II$ is isomorphic to a cube $(0\to 1)^n$, and a non-identity indecomposable morphism (called \emph{edge}) $J_1\leq J_2$ happens when either the subsequence $J_1$ is longer than $J_2$ by exactly one negative term (a \emph{negative edge}) or $J_2$ is longer than $J_1$ by exactly one positive term (a \emph{positive edge}). The subsequence consisting of all the positive terms in $I$, discarding the negative ones, is denoted by $I_+$. It is the maximum element of the poset $\II$.

To each $J\in\II$, thought of as a subsequence $J=(i_1,\dots,i_k)$,  there is associated a braid subword $w_J=\sigma_{i_1}\cdots\sigma_{i_k}$ of $w_I$. We also associate to $J$ a Soergel bimodule $B_J$ over the polynomial ring $\Z[x_1,\dots,x_r]$. From the topological perspective, this polynomial ring is the equivariant cohomology of a point $\htor^*=\HH^*(B\tT)$ over a rank $r$ torus $\tT$, and we keep denoting it this way. An explicit definition of $B_J$ is
\begin{equation}\label{Xform}B_J=\frac{\Z[X_{j,l}]_{0\leq j\leq n, \; 1\leq l\leq r}}{\left\langle \begin{array}{lr}X_{j,|l_j|}+X_{j,|l_j|+1}-X_{j-1,|l_j|}-X_{j-1,|l_j|+1},& J(j)=1\\ 
X_{j,|l_j|} X_{j,|l_j|+1}-X_{j-1,|l_j|} X_{j-1,|l_j|+1}, & J(j)=1\\
X_{j,l}-X_{j-1,l},&  l\neq |l_j|,|l_j|+1,\mbox{ or }J(j)=0\end{array}\right\rangle}\, ,\end{equation}
with $\HH^*_{\tT}$-bimodule structure afforded by ring maps
\begin{align*}
l:\htor^*&\to B_J & r:\htor^*&\to B_J\\
x_i&\mapsto X_{0,i}, & x_i & \mapsto  X_{n,i}.
\end{align*}
Henceforward, we denote the variables in $\htor^*$ by $x_i$ when acting on the left of a bimodule $B$ and $y_i$ when acting on the right. If the left or right module structure on a bimodule $B$ comes from a ring map $\htor^*\to B$, then we also write $x_i$ or $y_i$ for their corresponding images in $B$. Thus in the case of $B_J$ we have $x_i=X_{0,i}$ and $y_i=X_{n,i}$.

\begin{remark}
For $J=(i_1,\dots,i_k)$ a positive subword, an equivalent and more usual definition is
$$B_J=\htor^*\otimes_{(\htor^*)^{i_1}}\cdots\otimes_{(\htor^*)^{i_k}}\htor^*,$$
where $(\htor^*)^{i}\leq \htor^*$ is the subring of polynomials invariant under the permutation $x_i\leftrightarrow x_{i+1}$. This explains the relations of the form $X+Y-Z-W$ and $XY-ZW$ in (\ref{Xform}).
\end{remark}

A useful observation is that the combined left-right action map from the enveloping algebra $(\htor^*)^{\mathrm{en}}=\htor^*\otimes\htor^*\to B_J$ factors through the algebra $\coebs=\htor^*\otimes_{(\htor^*)^{\Sigma}}\htor^*$, where we tensor is over symmetric polynomials.

\begin{figure}[h!]\label{braiddiagram}
\centering\vspace{0.3cm}
\includegraphics{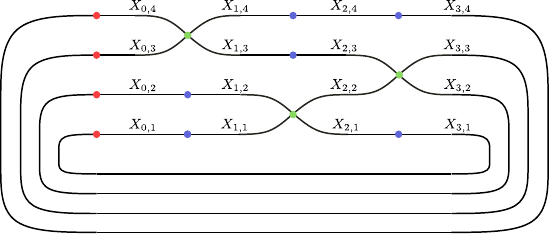}
\vspace{0.3cm}\caption{Singular diagram corresponding to the bimodule $B_{(3,1,2)}$ over $\Z[x_1,x_2,x_3,x_4]$.}
\end{figure}

It is practical to view (\ref{Xform}) as associated to a \emph{layered closed singular braid diagram} (\emph{singular diagram} in short) such as the one depicted in Figure \ref{braiddiagram}. Each vertex of a singular diagram lives in a layer $0\leq j \leq n$, numbered left to right. On each layer there is at most one 4-valent vertex, called a \emph{singular crossing} (green in figure), in which the strands $l_j$ and $l_j+1$ merge. All other vertices are 2-valent (red and blue), and we refer to these as \emph{marks}. Layer 0 will always consists of marks that will be referred to as \emph{closing marks} (red in figure). The incoming arcs to layer 0, which are the arcs that go around, are referred to as \emph{closing arcs}, and the corresponding variables are the \emph{closing variables}. These diagrams are always considered to be clock-wise oriented. Arcs have associated variables $X_{j,l}$, where $j$ indexes the layer of the source vertex and $l$ indexes the strand. If $D$ is a singular diagram, we can assign to it a ring $R(D)$ given by polynomials in the arc variables modded out by local relations. Each singular crossing contributes two relations like the top two in (\ref{Xform}), while each non-closing mark (blue in figure) contributes a relation like the bottom one.  Given a subsequence $J=(l_{j_1},\dots,l_{j_k})\subseteq I=(l_1,\dots,l_n)$ we can assign a singular diagram $D_J$ with layers $0\leq j \leq n$, where singular crossings merge contiguous strands $l_{j}$ and $l_{j}+1$ at the layer $j$ when $J(j)=1$. Layers $j$ in which $J(j)=0$ contain only marks. Then $B_J\cong R(D_J)$, where $x_i=X_{0,i}$ and $y_j=X_{n,i}$ are, respectively, the outgoing and incoming variable of the $i$-the closing mark. Notice that $x_i$ is the closing variable. 

So far, we have not considered relations corresponding to closing marks, which would be of the form $x_i-y_i$. Gearing towards the triply graded link homology categorifying the HOMFLY-PT polynomial, this is done in a derived fashion via Hochschild homology. More concretely, consider the enveloping algebra $(\htor^*)^{\mathrm{en}}=\Z [x_1,\dots,x_r,y_1,\dots,y_n]=\htor^*\otimes\htor^*$. Adding these closing relations on the nose would yield $$\hoch_0(B)=B\otimes_{(\htor^*)^{\mathrm{en}}}(\htor^*)^{\mathrm{en}}/(x_1-y_1,\dots x_r-y_r)$$ for any bimodule $B$ such as $B=B_J$, but we will rather consider the whole Hochschild homology $$\hoch(B)=B\otimes^{\mathbf{L}}_{(\htor^*)^{\mathrm{en}}}(\htor^*)^{\mathrm{en}}/(x_1-y_1,\dots x_r-y_r).$$ Since we are working over a polynomial ring, the Hochschild complex computing $\hoch(B)$ admits a concise model as a Koszul complex. This is a differential graded algebra
\newcommand{\choch}{\mathsf{CH}}
$$\choch(B):=B\otimes \Lambda(\hat x_1,\dots,\hat x_r)$$
with differential determined by $d_\mathsf{H}(\hat x_i)=x_i-y_i$ and the graded Leibniz rule with respect to the \emph{Hochschild} or \emph {Koszul grading} in which elements of $B$ have degree 0 and the exterior generators $\hat x_i$ have degree 1. When working with $B=B_J$ for positive $J$ we also have a \emph{cohomological grading} in which $X_{j,l}$ has degree 2 and $\hat x_i$ has degree 1.

\begin{definition}
Let $B$ be a $\htor^*$-bimodule. Its Hochschild complex is 
$$\choch(B):=B\otimes \Lambda(\hat x_1,\dots,\hat x_r)$$
with differential $d_\mathsf{H}$. Its Hochschild homology is
$$\hoch(B)=H_\mathsf{H}(\choch(B))=H(\choch(B),d_{\mathsf{H}}).$$
\end{definition}

These definitions of $\choch(B)$ and $\hoch(B)$ are functorial in $B$. The complex defining HOMFLY-PT homology is built by assigning a map $\hoch(B_{J_2})\to\hoch(B_{J_1})$ to each edge $J_1\leq J_2$. We refer to these as \emph{edge maps}, and they are defined as follows. 

\begin{definition}Given a positive edge $e$ of the form $J_1\leq J_2$ the \emph{positive edge map} $d_e:B_{J_2}\to B_{J_1}$ is induced by the identity map $\Z[X_{j,l}]\to\Z[X_{j,l}]$. For a negative edge $J_1\leq J_2$, where $J_1(j)>J_2(j)$, the \emph{negative edge map} $d_e:B_{J_2}\to B_{J_1}$ is induced by multiplication by $X_{j,|l_j|}-X_{j+1,|l_j|+1}$ in $\Z[X_{j,l}]$. There are corresponding positive and negative edge maps induced at the level of Hochschild complex and Hochschild homology.
\end{definition}
All induced maps always preserve the Hochschild degree. However, positive edge maps preserve cohomological degree, while negative edge maps increase it by 2.

\begin{figure}[h!]
\includegraphics{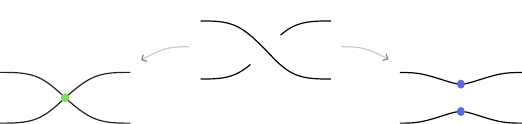}
\centering
\caption{Singular (or 0-) and oriented (or 1-) resolutions of a positive crossing. Numeration gets reversed for a negative crossing.}\label{resol}
\end{figure}

Let $L$ be a link presented by a braid word $w_I\in\Br(r)$ of length $n$. Each of the $n$ crossings in the braid can be resolved in two different ways as in Figure \ref{resol}, giving us $2^n$ possible singular diagrams.  Using the maps above we obtain a \emph{cubical diagram} 
\begin{align*}
\choch(B_\bullet) &: \II^{op} \to \mathsf{Ch}(\mathsf{Mod}_{\htor^*}).
\end{align*} By this we mean a graph morphism, where we regard $\II$ merely as a graph in which the only edges are the indecomposable morphisms as above. We obtain a double chain complex out of this cube by ``flattening it out'', which consists of arranging the Hochschild complexes corresponding to all possible subwords into
\begin{align}\label{rouquier}\begin{split}
\choch(B_{I_+})\longrightarrow\cdots \longrightarrow\bigoplus_{J\in \II^{t}\setminus\II^{t-1}}\choch(B_J)\{n_J\}\xlongrightarrow{d_v} &\bigoplus_{J\in \II^{t+1}\setminus\II^{t}}\choch(B_J)\{n_J\} \longrightarrow\\&\cdots \longrightarrow \choch(B_{I_-})\{n_I\}.\end{split}
\end{align}
Here we write $\II^t\subseteq \II$ for the subposet consisting of subsequences that are at most $t$ edges away from the maximal word $I_+$, so in this \emph{cubical grading} the elements of $\hoch(B_J)$ have degree equal to the distance from $J$ to $I_+$ in $\II$. The brackets $\{n_J\}$ express shifts in cohomological degree by $n_J$, twice the number of negative indices in $J$. This ensures that negative edge maps also preserve cohomological degree. We denote the entire cubical complex by
\begin{equation}\label{choch} CH(D):=\bigoplus_{J\in\II} \choch(B_J)\{n_{J}\}.\end{equation}
Given an edge $e$ of the form $J_1\leq J_2$, let $|e|$ be the amount of indices appearing in $J_1$ (or $J_2$) to the left of the index in which they differ. Then the differential $d_v$ in (\ref{rouquier}) is given by
$$d_v=\bigoplus_e (-1)^{|e|}d_e$$
running over all possible edges between subsequences of cubical degree $t$ and $t+1$. These signs are necessary to ensure that we have a chain complex. Since the maps $d_e$ come from bimodule maps, the differential $d_v$ commutes with $d_\mathsf{H}$.

\begin{theorem}\label{homflypt} (see \cite{krasnerintegral})
Let $L$ be a link represented by the closure of a braid diagram $D$. Up to overall degree shifts, the homology 
$$H_{v}H_{\mathsf{H}}(CH(D))=H_{v}\left(\bigoplus_{J\in\II} \hoch(B_J)\{n_J\}\right)$$
is a link invariant categorifying the HOMFLY-PT polynomial.
\end{theorem}

The differential $d_v$ in this complex can be referred to as \emph{Rouquier}, \emph{cubical} or \emph{vertical} differential. We stick with $v$ for vertical according to the notation in \cite{rasmussen}. In principle, $\choch(D)$ is a triply graded double complex of abelian group endowed with the cubical, Hochschild and cohomological gradings. For this reason, HOMFLY-PT homology is sometimes referred to as triply graded link homology, although our choice of degrees might not be the most standard. 

We end this subsection with the remark that each $\choch(B_J)$ is in fact an algebra over $\choch(\coebs)$, and in consequence $\hoch(B_J)$ is an $\hoch(\coebs)$-algebra. This structure is preserved by the differential $d_v$, and it is valuable for obtaining $\sln$ link homologies from the complex $\choch(D)$.

\subsection{The topological model}\label{topmodelpositive}
In this section we define torus-equivariant spaces that produce $B_J$ and $\hoch(B_J)$ for a braid word $w_J$ after taking Borel equivariant cohomology, and interpret the edge maps $B_{J_2}\to B_{J_1}$ geometrically.

Recall that if $G$ is a compact Lie group and $X$ is a $G$-space, the Borel $G$-equivariant cohomology of $X$ is defined as
$$\HH_G^*(X):=\HH(EG\times_G X),$$
where $\HH$ is singular cohomology with integral coefficients, $EG$ is any contractible right $G$-space with free action, and $EG\times_G X$ denotes the quotient of $EG\times X$ by the equivalence relation generated by $(eg,x)\sim(e,gx)$ for $e\in EG$, $x\in X$ and $g\in G$. The quotient space $EG\times_G X$ is sometimes known as the \emph{Borel construction} or the \emph{homotopy orbits} of the action, and is also denoted  $X_{hG}=EG\times_G X$. The homotopy orbits of the one point $G$-space $*$ constitute the classifying space $BG=*_{hG}$. Borel cohomology of $X$ has the structure of an algebra over $\HH_G^*:=\HH_G^*(*)=\HH^*(BG)$, and it is an invariant up to $G$-equivariant homotopy. An example of interest is a rank $r$ torus $G=\tT$. In such case a natural model is $E\tT=(S^\infty)^r$, so that $E\tT\times_\tT *=(S^\infty)^r/(S^1)^r=(\C P^\infty)^r$, meaning that the coefficient ring is $\HH_\tT^*=\Z[x_1,\dots, x_r]$ with $x_i$ of degree 2. For the unitary group $G=\Ur$, a model for $E\Ur$ is given the Stiefel manifold, which is the space of unitary $r$-frames in $\C^\infty$.

When $K\leq G$ is a closed subgroup, any $G$-free contractible space $EG$ is also an $EK$ after restricting the action. So, for instance, $r$-frames in $\C^\infty$ is also a model for $E\tT$. One consequence is that, given any $K$-space $X$, we have an isomorphism $\HH_K^*(X)=\HH_G^*(G^{\mathsf{r}}\times_K X)$, where the superscript $G^{\mathsf{r}}$ indicates that the $K$-action by right multiplication is being used. Another consequence is that the Borel cohomology $\HH_K(G/K)$ of the homogeneous space $G/K$, considered as a $K$-space by left multiplication, has the structure of an $\HH_K^*$-bimodule. This is obtained from two maps $$\lambda,\rho:EG\times_K G/K\to EG\times_K *$$ inducing the left and right $\HH_K^*$-algebra structures, respectively. Namely, the map $\lambda$ is induced by the unique map $G/K\to *$, and it gives the standard $\HH_K^*$ module structure, while the map $\rho$ sends $[e,g]\mapsto[eg,*]$. Notice that $\rho$ makes use of the whole $G$-action on $EG$.

The latter observation is relevant for us in the particular case of $\Ur$ and its standard maximal torus $\tT$. The homogeneous space $\Ur/\tT$ is the usual flag variety, the moduli space of complete flags in $\C^r$. Various characterizations of its equivariant cohomology $\htor^*(\Ur/\tT)$ as a ring and as an $\htor^*$-bimodule are elaborated in the Appendix. Among these, there is an identification of $\htor^*(\Ur/\tT)$ with $\coebs=\htor^*\otimes_{(\htor^*)^\Sigma}\htor^*$.

To build our spaces of interest we will be looking at various subgroups of $\Ur$. Denote by $G_i\leq\Ur$ the subgroup consisting of block diagonal unitary matrices having a 2 by 2 block in positions $i$, $i+1$, and 1 by 1 blocks elsewhere along the diagonal. The torus $\tT$ still acts on each $G_i$ by both left and right multiplication. Given $J=(i_1,\dots,i_k)$, first define $$\BB_{\tT\times\tT}(w_J):=G_{i_1}\times_\tT\cdots\times_\tT G_{i_k}.$$ This is a $\tT\times\tT$-space in which one copy of $\tT$ acts by left multiplication on $G_{i_1}$, and the other $\tT$ acts by right multiplication on $G_{i_k}$. Next we obtain $\tT$ spaces out of $\BB_{\tT\times\tT}(w_J)$ in two different ways. On the one hand, we may restrict the action along the inclusion of a diagonal subtorus $\iota:\tT=\lbrace (t,t^{-1})\rbrace\hookrightarrow \tT\times\tT$ and get
$$\BB_{\tT}(w_J):=\iota^*(\BB_{\tT\times\tT}(w_J))=G_{i_1}\times_\tT\cdots\times_\tT G_{i_k},$$
so that the space is the same but the action is now by conjugation in the first and last factor. On the other hand, we can mod out the right action to obtain
$$\BTS_{\tT}(w_J):=\BB_{\tT\times\tT}(w_J)/\tT=G_{i_1}\times_\tT\cdots\times_\tT G_{i_k}/\tT,$$
so the only $\tT$-action is on the first factor. These two spaces are related by the projection
$$\BB_{\tT}(w_J)\to\BTS_{\tT}(w_J),$$
which is a $\tT$-equivariant principal $\tT$-bundle. This extends to a commutative square
$$\begin{tikzcd}\BB_{\tT}(w_J)\ar[r]\ar[d]&\BTS_{\tT}(w_J)\ar[d]\\ \Ur^c \ar[r]&\Ur/\tT
\end{tikzcd}$$
where vertical maps are both given by multiplication $[g_1,\dots,g_k]\mapsto [g_1\dots g_k]$, and the bottom map is projection. The small $c$ in $\Ur^c$ says that the $\tT$-action is by conjugation. All maps are $\tT$-equivariant and this square is in fact a morphism of principal $\tT$-bundles.

The map $\BTS_\tT(w_J)\to\Ur/\tT$ then induces an $\coebs$-algebra structure $\coebs\to\htor^*(\BTS_\tT(w_J))$, which in turn yields an $\htor^*$-bimodule structure on $\htor^*(\BTS_\tT(w_J))$ by precomposition with the two maps $\htor^*\to\coebs$. Now, another description of $\coebs$ (see Appendix) is as a free left $\htor^*$-module on the Schubert classes $S_w$ indexed by the Weyl group, which is the symmetric group $\Sigma_r$ of permutations of $r$-elements. The right $\htor^*$-action is explicitly given by $y_i= x_i+S_{\sigma_i}-S_{\sigma_{i-1}}$. 

\begin{proposition}
Let $w_J$ be a positive braid word. Then the $\htor^*$-bimodule $\htor^*(\BTS_\tT(w_J))$ is isomorphic to the Soergel bimodule $B_J$.
\end{proposition}
\vspace{-0.3cm}
\begin{proof} An explicit formula, given in a recursive form in \cite{nitubroken2} through a classical argument of Bott and Samelson, is
\begin{equation}\label{deltaform}\htor^*(\BTS_\tT(w_J))=\frac{\Z[x_1,\dots,x_r,\delta_1,\dots,\delta_k]}{\left\langle \delta_t^2 + \left(x_{i_t}-x_{i_{t+1}}+ 2 \sum_{s<t,\, i_s=i_t} \delta_s -\sum_{s<t,\, |i_s-i_t|=1} \delta_s \right)\delta_t\right\rangle}\, ,\end{equation}
with left and right actions given respectively by
\begin{align*}
l:\htor^*&\to \htor^*(\BTS_\tT(w_J)) & r:\htor^*&\to \htor^*(\BTS_\tT(w_J))\\
x_i&\mapsto x_i, & y_i & \mapsto y_i= x_i +  \hat\delta_i-\hat\delta_{i-1},
\end{align*}
 where $\hat\delta_i=\sum_{i_t=i} \delta_j$ is the image of $S_{\sigma_i}$ under the multiplication map. An $\htor^*$-bimodule (and algebra) isomorphism onto $B_J$ as prescribed in the formula (\ref{Xform}) is defined on generators as follows: if $J=(i_1,\dots,i_k)=(l_{j_1},\dots,l_{j_k})$ as a subword of $I$, then we map $\delta_t\mapsto X_{j_t,i_t}-X_{j_t-1,i_t}$ .
\end{proof}
 
\begin{remark}
Under the isomorphism in the proof, the expression in parenthesis in (\ref{deltaform}) is
 $$x_{i_t}-x_{i_{t+1}}+ 2 \sum_{s<t,\, i_s=i_t} \delta_s -\sum_{s<t,\, |i_s-i_t|=1} \delta_s=X_{j_t-1,i_t}-X_{j_t-1,i_t+1}.$$ Using the notation $$[\alpha_i]_t:=X_{j_t-1,i}-X_{j_t-1,i+1},$$ the relations in the presentation $(\ref{deltaform})$ take the compact form $(\delta_t+[\alpha_{i_t}]_t)\delta_t=0$.
\end{remark}

For a positive edge $J_1\leq J_2$, say $i_t>0$, $J_2=(i_1,\dots,i_t,\dots,i_k)$ and $J_2=J_1\setminus i_t$, the map $B_{J_2}\to B_{J_1}$ appearing in the Rouquier complex (\ref{rouquier}) has a simple topological lift: from the inclusion $\tT\hookrightarrow G_{i_s}$ we induce
 $$\BTS_\tT(w_{J_1})=G_{i_1}\times_{\tT}\cdots\times_\tT\tT\times_\tT\cdots G_{i_k}\xlongrightarrow{\iota} G_{i_1}\times_{\tT}\cdots\times_\tT G_{i_t}\times_\tT\cdots \times_{\tT}G_{i_k}=\BTS_\tT(w_{J_2}).$$ The $\htor^*$-algebra map $$\iota^*:\htor^*(\BTS_\tT(w_{J_2}))\to \htor^*(\BTS_\tT(w_{J_1}))$$ in cohomology is easily described as killing $\delta_{t}$ while preserving all the other variables. This map coincides with the positive edge map $B_{J_2}\to B_{J_1}$ through the relevant isomorphism.
 
 When $J_1\leq J_2$ is a negative edge, say $i_t<0$, $J_1=(i_1,\dots,i_t,\dots,i_k)$ and $J_2=J_1\setminus i_t$, we get $B_{J_2}\to B_{J_1}$ by starting with $\tT\hookrightarrow G_{i_k}$ and inducing from an inclusion $\iota:\BTS_\tT(w_{J_2})\to\BTS_\tT(w_{J_1})$. In this case, we have a wrong-way map $\iota_*$ in cohomology that can be defined via the Thom-Pontryagin construction applied to the normal bundle of the inclusion $\iota$. More specifically, we have a map of $\htor^{*}(\BTS_\tT(w_{J_1}))$-modules (in particular $\coebs$-modules) $$\iota_*:\htor^{*}(\BTS_\tT(w_{J_2}))\to \htor^{*+2}(\BTS_\tT(w_{J_1}))$$ 
that is completely determined by
$$\iota_*(1)=\delta_t+[\alpha_{i_t}]_t,\qquad \iota_*(v\cdot\iota^*(u))=\iota_*(v)\cdot u.$$
These conditions are sufficient to identify $\iota_*$ with the negative edge map $B_{J_2}\to B_{J_1}$.

To give an interpretation of the Hochschild complex $\choch(B_J)$ and the respective Hochschild homology, we turn our attention to the cohomology of $\BB_\tT(w_J)$. We summarize the relevant results in \cite{nitubroken2} in the following proposition. For convenience, when the notation for a space already indicates the acting group in a subindex, like in $\BB_\tT$ or $\BTS_\tT$, we simply add the lowercase $h$ to indicate Borel construction.

\begin{proposition}\label{serressprop} The Serre spectral sequence associated to the fibration $$\tT\to \BB_{h\tT}(w_J)\to \BTS_{h\tT}(w_J)$$ has as its $E_2$-page the Koszul complex that computes Hochschild homology of the bimodule $B_J=\htor^*(\BTS_\tT(w_J))$. It collapses at $E_3$, where it is free over $\Z$. In consequence, $\htor^*(\BB_\tT(w_J))$ is non-canonically isomorphic to $\hoch(B_J)$ as abelian groups.
\end{proposition}

More explicitly,
\begin{equation}\label{serress}E_2=\htor^*(\BTS_\tT(w_J))\otimes\Lambda(\hat x_1,\dots,\hat x_r)\, , \end{equation}
where the $\hat x_i$ form the standard generators for $\HH^1(\tT)$, and the differential, being a derivation with respect to the exterior $\htor^*(\BTS_\tT(w_J))$-algebra structure, is completely determined by
$$d_2(\hat x_i)= x_i-y_i=\hat \delta_{i-1} -\hat\delta_i.$$

The identification $\htor^*(\BB_\tT(w_J))\cong\hoch(B_J)$ as abelian groups is non-canonical over $\Z$ but can be made canonical over $\Q$. As $\htor^*$-modules, we have an isomorphism between $\hoch(B_J)$ and the page $E_\infty=E_3$ in the spectral sequence, which is an associated graded $\gr_{Serre}\left[ \htor^*(\BB_\tT(w_J))\right]$ with respect to the \emph{Serre filtration} $\{\mathcal F^\bullet \htor^*(\BB_\tT(w_J))\}$. This filtration is induced from the cellular filtration of the base $\BTS_{h\tT}(w_J)$. The base-fiber bigrading in $E_3^{*,*}$ coincides with the Soergel-Hochschild bigrading in $\hoch(B_J)$. 

We observe that given positive or negative edges $J_1\leq J_2$, we still have corresponding inclusions $\iota:\BB_\tT(w_{J_1})\to \BB_\tT(w_{J_2})$ (reversed in the negative case) induced from $\tT\hookrightarrow G_{i_t}$. The induced maps $\iota^*$ or $\iota_*$ still give the correct positive and negative edge maps at the level of Hochschild homology.

\subsection{Spectrum of strict broken symmetries and INS conditions} 

Let $w_I$ be a braid word in $r$ strands, with possibly negative indexes, and let $\II$ be its poset of subwords.
The discussion from the previous section expresses the cubical diagram $\hoch(B_\bullet):\II^{op}\to\mathsf{Mod}_{\htor^*}$ in terms of Borel equivariant cohomology of spaces $\BB_\tT(w_J)$ and maps between them.

In \cite{nitubroken1}, Kitchloo defines a \emph{filtered $\Ur$-equivariant spectrum of strict broken symmetries} $s\BB(L)$ for the link $L$ represented by the closure of $w_I$. It is built from homotopy colimits of the cubical diagram $\BB_\tT(w_\bullet)$. With adequate suspensions and desuspensions, analogous to the overall shifts usually needed in the construction of algebraic link invariants, the filtered object $s\BB(L)$ is itself an invariant of $L$ up to a notion of \emph{quasi-equivalence}, a sort of chain homotopy equivalence at the level of associated graded objects.

One may apply equivariant cohomology theories to the filtered object $s\BB(L)$ to obtain algebraic link invariants. In the case of Borel equivariant (singular) cohomology, the complex  $(\bigoplus_J\hoch(B_J)\{n_J\},d_v)$ appears as the $E_1$-page of the spectral sequence associated to the filtration of $s\BB(L)$, so that $E_2$ is HOMFLY-PT homology and all further pages are also link invariants. In order for a more general cohomology theory $\E^*$ to produce algebraic invariants following this same recipe, some formal conditions related to braid and Markov moves should be met.

To make this precise, let $\E^*$ be a generalized cohomology theory (possibly twisted in the sense of Section \ref{twisted}). The homology of the complex $\bigoplus_J\E_\tT^{*+n_J}(\BB_\tT(w_J))$ is a link invariant (up to overall shift) provided that the maps below are injective, null or surjective  (abbreviated \emph{INS}) for all $r$ and all braid words $w_J\in\Br(r)$ of the indicated forms. We refer to these as \emph{INS conditions}.
\begin{enumerate}
\item For a positive edge $J_1\leq J_2$ with the sequence $J_1=(i_1,\dots,i_{l-1})$ not containing either $\pm(r-1)$ and $J_2=(i_1,\dots,i_{l-1},r-1)$, the inclusion $\BB_\tT(w_{J_1})\to\BB_\tT(w_{J_2})$ induces an INS homomorphism $\iota^*$ in $\E_\tT^*$.\\

\item For a negative edge $J_1\leq J_2$ with the sequence $J_2=(i_1,\dots,i_{l-1})$ not containing either $\pm(r-1)$ and $J_1=(i_1,\dots,i_{l-1},-(r-1))$, the inclusion $\BB_\tT(w_{J_2})\to\BB_\tT(w_{J_1})$ induces an INS (wrong-way) homomorphism $\iota_*$ in $\E_\tT^*$.\\

\item For positive $s$ and a sequence $J=(i_1,\dots,i_{l-3},s,s-1,s)$ or $J=(i_1,\dots,i_{l-3},s-1,s,s-1)$, let $\mathcal{X}=G_{i_1}\times_{\tT}\cdots\times_{\tT}G_{i_{l-3}}\times_{\tT}  (\tT^{s-2}\times  \mathrm U (3)\times \tT^{r-s-1})$. Then the map $\BB_\tT(w_J)\to \mathcal{X}$ that does multiplication in the last three components induces an INS homomorphism in $\E^*_\tT$.\end{enumerate}

In the particular case of singular cohomology $\HH^*$, induced homomorphisms of all forms (1), (2) and (3) are injective. We will return to INS conditions for other theories in later sections.

\section{The Khovanov-Rozansky complex and Rasmussen spectral sequences}

In this section we review Khovanov-Rozansky homology via Koszul (exterior) complexes in such a way that it becomes easily comparable with the geometric formalism. The comparison will be made explicit in subsequent sections.

 \subsection{Khovanov-Rozansky link homologies from Koszul complexes}  Khovanov and Rozansky (\cite{krmatrix}) resolve link diagrams in a similar way to the cubes of singular diagrams already discussed. Associated to each resolution $J$ there is a Koszul complex $(K_p(D_J),d_+)$ with Koszul generators prescribed by the various crossings and marks in the singular diagram $D_J$. There is, however, a second differential $d_-$ in this Koszul complex given by multiplication by a certain $d_+$-cycle, the choice of which relates to the subindex $p$. Put together, $(K_p(D_J),d_+,d_-)$ forms a double complex. Now considering all resolutions of a braid diagram at the same time, there is also a cubical differential $d_v$ which commutes with $d_+$ and $d_-$, and thus one has the entire Khovanov-Rozansky triple complex
$$KhR_p(D)=\left(\bigoplus_{J\in\II} K_p(D_J),d_+,d_-,d_v\right).$$
 The more general formalism of \emph{matrix factorizations} features eminently in the original references. These can be thought of as double complexes in which the differentials need not anti-commute but instead they satisfy an equation $d_+d_-+d_-d_+=w\cdot\mathrm{Id}$. We will avoid such language because we deal only with closed diagrams, in which case $w=0$ and therefore $(K_p(D_J),d_+,d_-)$ is an actual double complex. Most formulas and calculations will be described using the Koszul complex structure on $K_p(D_J)$. We generically refer to versions of this construction $KhR_p(D)$ as \emph{KR-complexes}. 
 
 The construction depends on the choice of a graded commutative coefficient ring $\R$ and a potential $p(X)\in\R[\![X]\!]$. The variable $X$ here has cohomological degree 2, and power series are taken in the graded sense so that $\R[\![X]\!]=\R[X]$ except only possibly when $\R$ is non-trivial in negative degrees. The coefficient rings $\R=\Z, \Q$ are thought of as living in degree 0. The original Khovanov-Rozansky $\sln$-link homology categorifying the $\sln$-link polynomial is obtained by taking homologies of $KhR_p(D)$ in the case $p(X) = X^{n+1}\in \Q[X]$. If one dispenses with the differential $d_-$, the other two differentials $d_+$ and $d_v$ produce HOMFLY-PT homology with coefficients extended into $\R$.  A primary goal of this section is to review how the KR-complex can be reduced, in a sense, to the complex $CH(D)$ that we previously had for HOMFLY-PT homology.  The differentials $d_+$ and $d_v$ in $KhR_p(D)$ correspond to the Hochschild and cubical differentials $d_\mathsf{H}$ and $d_v$ in $CH(D)$, and we keep track of the behavior of the extra differential $d_-$ under such reduction. The precise relation between $\sln$ and HOMFLY-PT is the so called Rasmussen spectral sequence, which will be also presented.

\begin{figure}[h!]
\centering\vspace{0.3cm}
\includegraphics{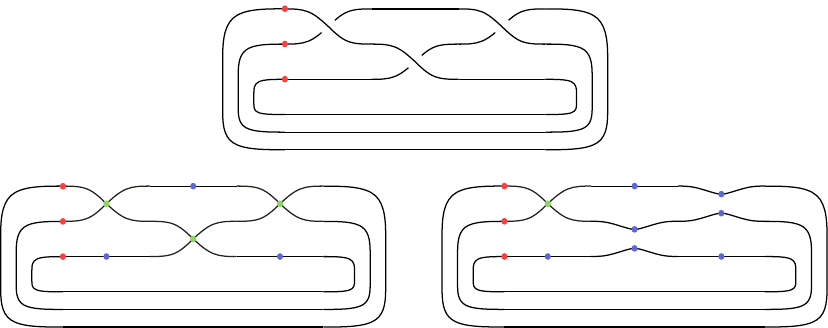}\vspace{0.3cm}
\caption{A closed braid diagram on top. On bottom we have closed singular diagrams, which correspond to the resolutions $s=000$ and $s'=011$, out of eight possible. See Figure \ref{resol} below for resolutions at a single crossing.}
\label{closeddiags}
\end{figure}

Let us begin with the definition of $K_p(D_J)$ for a singular diagram $D_J$. We will make use of the graded power series ring $\R[\![X_e]\!]$ with one variable of cohomological degree 2 for each edge $e$ of the diagram. It is convenient to use the labeling scheme  indicated in Figure \ref{vertexlabel} when working locally around a mark $m$ or a crossing $c$. This stands in contrast with the previously established global labeling by layer and strand.

 \begin{figure}[h!]
 \includegraphics{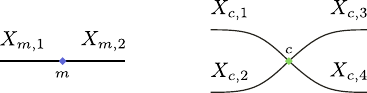}
\centering\caption{Local labeling for edge variables adjacent to a vertex.}\label{vertexlabel}
\end{figure}\phantom{a}\\
Let us fix notation for three kinds of elements in $\ringD$ which correspond to the three kind of relations associated to marks and singular crossings in the Soergel bimodule. These are 
$$\Gamma_m:= X_{m,2}-X_{m,1}$$ for each mark $m$, as well as $$\Upsilon_c:=X_{c,3}+X_{c,4}-X_{c,1}-X_{c,2}\quad\mbox{and} \quad\Xi_c:=X_{c,3}X_{c,4}-X_{c,1}X_{c,2}$$ for each singular crossing $c$.

\let\oldgamma\gamma
\renewcommand{\gamma}{\bm{\oldgamma}}
\let\oldxi\xi
\renewcommand{\xi}{\bm{\oldxi}}
\let\oldupsilon\upsilon
\renewcommand{\upsilon}{\bm{\oldupsilon}}
\let\oldkappa\kappa
\renewcommand{\kappa}{\bm{\oldkappa}}

Define $K_p(D_J)$ be the Koszul complex of all the $\Gamma_m$, $\Xi_c$  and $ \Upsilon_c$ with coefficients in $\ringD$, where $m$ runs through all the marks and $c$ runs through all the singular crossings in $D_J$. The Koszul differential is denoted by $d_+$, so we have that
$$K_p(D_J)=\ringD\otimes \Lambda(\gamma_m,\xi_c,\upsilon_c)$$
is an $\ringD$-algebra in which the Koszul generators $\gamma_m$ and $\upsilon_c$ have cohomological degree 1, while $\xi_c$ has cohomological degree 3. The differential $d_+$ is a derivation (in the graded sense) determined by
$$d_+(\gamma_m)=\Gamma_m,\quad d_+(\xi_c)= \Xi_c,\quad d_+(\upsilon_c) = \Upsilon_c.$$

The second differential $d_-$ depends on the choice of homogeneous potential $p(X)$ in $\R[\![ X]\!]$. Let $\pi$, $u_1$ and $u_2$ be multi-variable power series with coefficients in $\R$ satisfying
\begin{align*}
p(X)-p(Y)  &= (X-Y)\pi(X,Y), \\
p(X)+p(Y)-p(Z)-p(W)  = & (X+Y-Z-W)u_1(X,Y,Z,W)\\&+(XY-ZW)u_2(X,Y,Z,W).
\end{align*}

Whenever variables can be inferred from the context, we omit inputs for $u_1$ and $u_2$ and $\pi$, or indicate them in superscripts. In light of this, consider the $d_+$-cycle in $K_p(D_J)$ given by
$$\kappa_J:=\sum_m \pi^m\gamma_m+\sum_c u^c_1\upsilon_c+u^c_2\xi_c,$$
with the first sum running over all matchings $m$ and the second running over all crossings $c$. Let $d_-:K_p(D_J)\to K_p(D_J)$ be defined by $d_- (f)=\kappa_J\wedge f$. Notice that $d_-$ anticommutes with $d_+$ by virtue of $\kappa_J$ being a $d_+$-cycle.  We also remark that $\kappa_J$ is homogeneous both in cohomological and Koszul grading.

Now, given a link presented as the closure of a braid word $w$, we want to produce a cube of resolutions $$K_p(D_\bullet):\II^{op}\to\Ch(\ringD).$$
 Let us describe the edge maps. If an edge $J_1\leq J_2$ is positive, then generators $\xi_c$ and $\upsilon_c$ corresponding to a singular crossing disappear, and in turn we get generators $\gamma_{m_1}$ and $\gamma_{m_2}$ corresponding to two new markings. For a negative crossing we have the reciprocal situation. The total number of edges does not change, and the local picture at the crossing involves four variables which could be labeled either by $X_{c,1},X_{c,2},X_{c,3},X_{c,4}$ or $X_{m_1,1},X_{m_2,1},X_{m_1,2},X_{m_2,2}$ depending on whether we look at the singular or non-singular resolution. For convenience in the current definition, let us label these variables uniformly by $X_1,X_2,X_3,X_4$.
Both complexes $K_p(D_{J_1})$ and $ K_p(D_{J_2})$ contain the sub-algebra $A$ generated by $\ringD$ and all the $\gamma_m$, $\upsilon_c$ and $\xi_c$ not involved in the crossing, so that $K_p(D_{J_1})$ and $ K_p(D_{J_2})$ are both rank 4 free modules over $A$.  The positive edge maps $d_e:K_p(D_{J_1})\to K_p(D_{J_2})$ are determined as $A$-module maps by
\begin{align*}1&\mapsto 1+ \frac{u_1+X_1u_2-\pi^{2,4}}{X_{3}-X_{1}}\gamma_{m_1}\wedge\gamma_{m_2}\\
\upsilon_c&\mapsto \gamma_{m_1}+\gamma_{m_2}\\
\xi_c&\mapsto X_4\gamma_{m_1} +X_1\gamma_{m_2}\\\upsilon_c\wedge\xi_c&\mapsto (X_1-X_4)\gamma_{m_1}\wedge\gamma_{m_2}.
\end{align*}
The negative edge maps $d_e:K_p(D_{J_1})\to K_p(D_{J_2})$ are determined as $A$-module maps by
\begin{align*} 1&\mapsto X_{1}-X_{4}-\frac{u_1+X_1u_2-\pi^{2,4}}{X_{3}-X_{1}}\upsilon_c\wedge\xi_c\\ \gamma_{m_1}&\mapsto X_1\upsilon_c-\xi_c \\
\gamma_{m_2}&\mapsto\xi_c-X_4\upsilon_c\\
\gamma_{m_1}\wedge\gamma_{m_2}&\mapsto \upsilon_c\wedge\xi_c.
\end{align*}
These maps commute with both $d_+$ and $d_-$. These edge maps assemble into a cubical differential $d_v$ just as in the previous section.

\begin{definition}
Let $D$ be a braid diagram, $\R$ a graded commutative ring and $p(X)\in\R[\![X]\!]$ a homogeneous potential. The \emph{KR-complex} $KhR_p(D)$ is the flattening of the cube $K_p(D_\bullet)$ into 
$$KhR^\R_p(D)=\bigoplus_{J\in\II} K_p(D_J)\{n_J\},$$
with anti-commuting differentials $d_+$, $d_-$ and $d_v$. When $\R=\Q$, we omit it from the notation.
\end{definition}

Let us record how the KR-complex computes $\sln$-link homology and HOMFLY-PT homology. Since we are dealing with triple complexes, we lighten the notation for homologies by writing $H_+=H(-,d_+)$ and similarly for $H_-$ and $H_v$. We also have total homologies $H_{+-}=H(-,d_++d_-)$ and $H_{v-}=H(-,d_v+d_-)$.

\begin{theorem}\label{linkhomologies}
Let $\R=\Q$ and $p(X)=X^{n+1}$, and let $L$ be a link presented by the closure of a braid diagram $D$. 
\begin{enumerate}
\item \cite{krmatrix} After adequate degree shifts, the homology 
$$H_{v}H_{+-}(KhR_p(D)),$$
is a bigraded link invariant categorifying the $\sln$ polynomial.
\item \cite{krmatrix} After adequate degree shifts, the homology 
$$H_{v}H_+(KhR_p(D)),$$
is a triply graded link invariant categorifying the HOMFLY-PT polynomial.
\item \cite{rasmussen} There are canonical isomorphisms
$$H_{v}H_{+-}(KhR_p(D))\cong H_{v-}H_{+}(KhR_p(D))\cong H_{v}H_{-}H_{+}(KhR_p(D))$$
giving alternative definitions of (1) from the same complex.
\end{enumerate}
\end{theorem}

Literal (2) is already suggestive of the relation between $(KhR^\Z_p(D),d_+,d_v)$ and the complex $(CH(D),d_{\mathsf{H}},d_v)$. Given a potential $p(X)\in\Z[X]$, one defines the corresponding differential $d_-$ in $CH(B_J)$ by multiplication by the Hochschild cycle 
$$\kappa'_J=\sum_{j=1}^r \frac{p(x_j)-p(y_j)}{x_j-y_j}\hat x_j.$$
Notice that this formula does not depend on $J$, and in fact it can be seen as coming from a cycle $\kappa'\in\choch(\coebs)$. Since the $\choch(\coebs)$-algebra structure on the various $\choch(B_J)$ is respected by the edge maps $d_e$ that constitute $d_v$, this implies that $d_v$ and $d_-$ anti-commute.

\begin{proposition}\label{coeffshomflypt} 
Let $\R=\Z$ and $p(X)\in \Z[X]$, and let $D$ be a closed braid diagram for the link $L$. There is a morphism of triple complexes
$$(KhR^\Z_p(D),d_+,d_-,d_v)\to(CH(D),d_{\mathsf{H}},d_-,d_v)$$
which induces at each vertex an isomorphism $H_{+}(K_p(D_J))\cong \hoch(B_J)$. Therefore, as double complexes,
$$\left(H_+(KhR^\Z_p(D)),d_-,d_v\right)=\left(\bigoplus_{J\in\II}\hoch(B_J)\{n_J\},d_-,d_v\right)$$
\end{proposition}

\begin{proof}
This proof is done with more detail in \cite{khovanovtriply}. For a given resolution $D_J$, consider the sequence $\mathcal S$ formed by all $\Gamma_m$, $\Xi_c$ and $\Upsilon_c$, except for those $\Gamma_m$ corresponding to closing marks. The sequence $\mathcal S$ is regular. The ring $\Z[X_e]/\langle\mathcal{S}\rangle$ is isomorphic to the Soergel bimodule $B_{J}$, which is directly evidenced in the presentation (\ref{Xform}). This is a bimodule isomorphism when we take the left $\Z[x_1,\dots,x_r]$-module structure given by $x_i\mapsto X_{m_i,1}$ and the right module structure given by $x_i\mapsto X_{m_i,2}$, where $m_1,\dots,m_r$ are the closing marks. The quotient $\Z[X_e]\to B_J$ extends to a map of complexes
$$(K_p(D_J),d_+)\mapsto (\choch(B_J),d_{\mathsf{H}}),$$
with $\gamma_{m_i}\mapsto \hat x_i$ and all other exterior generators mapping to 0. This is a quasi-isomorphism by the regularity of $\mathcal S$. Thus $H_{+}(D_J)\cong\hoch(B_J)$.
When the Koszul generators coming from crossings map to 0, the edge maps $d_e$ in $KhR^\Z_p(D)$ simplify into those of $CH(D)$, which guarantees preservation of $d_v$. Moreover, the formula for $\kappa_J$ simplifies into the formula for $\kappa'_J$, which guarantees the preservation of $d_-$. \end{proof}

In view of this result, we replace the notation for the Hochschild differential $d_{\mathsf{H}}$ in $CH(D)$ for $d_+$. 

It is an interesting consequence of the Proposition that $d_v$ becomes homogeneous after taking $H_+$ homology. More precisely, the explicit formulas given above for positive and negative edge maps $d_e:K_p(D_{J_1})\to K_p(D_{J_2})$ do not respect cohomological or Koszul gradings due to the non-homogeneous terms in $d_e(1)$. Those terms vanish after applying the quasi-isomorphism $(K_p(D_J),d_+)\mapsto (\choch(B_J),d_{\mathsf{H}})$.

\begin{remark}
The Proposition can also be formulated over $\Q$ if one changes coefficients $CH(D)\otimes \Q$. As a matter of fact, the Hochschild homology $\hoch(B_j)$ is free over $\Z$ and
$$H_{v}H_{+}(CH(D)\otimes \Q)\cong H_{v}H_+(CH(D))\otimes \Q$$
is the same HOMFLY-PT homology computed by the KR-complex.
\end{remark}

We can now summarize the rational link homologies in the following definition.

\begin{definition}\label{deflinkhomologies}
Let $L$ be a link presented by a braid diagram $D$, and let $p(X)=X^{n+1}\in\Q [X]$. The \emph{Khovanov-Rozansky $\sln$ link homology} of $L$ is
$$KRH_n(L):=H_{v}H_{-}H_{+}(KhR_{p}(D))=H_{v}H_{-}H_{+}(CH(D))$$
The \emph{Khovanov-Rozansky triply graded link homology} or \emph{HOMFLY-PT homology} of $L$ is
$$KRH_\infty(L):=H_{v}H_{+}(KhR_{p}(D))=H_{v}H_{+}(CH(D)).$$
The latter is independent of the choice of $n$. 
\end{definition}

So far we have been working with the cubical, Hochschild and cohomological grading in $CH(D)$. The way we just defined them, these link homologies are only link invariants after overall grading shifts. An argument of Rasmussen, which we touch on in the following section, shows that the Hochschild grading collapses into a single level in $H_{-}H_{+}(CH(D))$. This is the reason why $KRH_n(L)$ is considered a bigraded invariant.

\subsection{Rasmussen spectral sequence and integral $\sln$ homology}
We have seen how integral HOMFLY-PT homology $H_{v}H_+(CH(D))$ is a link invariant after appropriate shifts. In view of Theorem \ref{linkhomologies} and Definition \ref{deflinkhomologies}, one could expect $H_{v}H_{-}H_+(KhR^\Z_p(D))$ or $H_{v-}H_+(KhR^\Z_p(D))$ to give integral $\sln$ link homology invariants for $p(X)=X^{n+1}\in\Z [X]$. 

A fruitful approach first taken by Rasmussen in \cite{rasmussen} over $\Q$ is to think about the spectral sequences associated to the double complex $(H_+(KhR_p(D)),d_-,d_v)$ for a link $L$ with diagram $D$. Let us consider the bigrading $(\gr_-,\gr_v)$ by Hochschild and cubical degree. Here $d_-$ can be thought of as a horizontal differential of degree $(1,0)$ and $d_v$ is the vertical differential of degree $(0,1)$. Filtering the complex horizontally affords one possible spectral sequence with signature
$$\prescript{I}{}E_1^{*,*}(D)=H_-H_+(KhR_p(D))\Longrightarrow H_{v-}H_+(KhR_p(D)).$$
Rasmussen proves that this spectral sequence collapses at the second page. In fact, $\prescript{I}{}E_1^{*,*}$ already lives in top Hochschild degree $r$, meaning that the spectral sequence is concentrated in a single line $\prescript{I}{}E_1^{r,*}$. Since $d_s$ has bidegree $(s,1-s)$ it follows that $d_1=d_v$ is the last non-trivial differential. The spectral sequence collapses at the second page, giving an isomorphism $\prescript{I}{}E_2^{*,*}=H_vH_-H_+(KhR_p(D))\cong H_{v-}H_+(KhR_p(D))$.

The vertical filtration gives another spectral sequence
$$\prescript{II}{}E_1^{*,*}(D)=H_vH_+(KhR_p(D))\Longrightarrow H_{v-}H_+(KhR_p(D)).$$
This is ordinarily referred to as the \emph{Rasmussen spectral sequence}. In the case $p(X)=X^{n+1}\in \Q[X]$, the Rasmussen spectral sequence relates the rational homologies $KRH_\infty$ and $KRH_n$. Notice that, since this is all over the field $\Q$, the extension problems are trivial and thus $\prescript{II}{}E_\infty^{*,*}(D)\cong H_{v-}H_+(KhR_p(D))=KRH_n(L)$. 

Integrally, one can start with the double complex $(H_+(KhR^\Z_p(D)),d_-,d_v)$ and use the fact that the integral counterpart of $\prescript{II}{}E_1^{*,*}$, denoted by $\prescript{II}{\Z}E_1^{*,*}$, is already a link invariant to conclude that all subsequent pages, as well as $H_{v-}H_+(KhR^\Z_p(D))$, are also link invariants. This is the approach taken by Krasner in \cite{krasnerintegral}, which leads him to define integral $\sln$-link homology as $\prescript{II}{\Z}E_\infty^{*,*}(D)$.

Over $\Z$, however, there may be non-trivial extension problems preventing the existence of isomorphisms $\prescript{I}{\Z}E_2^{*,*}(D)\cong H_{v-}H_+(KhR^\Z_p(D))\cong \prescript{II}{\Z}E_\infty^{*,*}(D)$. Part of the issue is the presence of integral torsion even in the most basic examples. To illustrate, let us compute the homology of the unknot represented by the trivial braid word $1\in\Br(1)$. We use the complex $CH(D)$ associated to the trivial diagram $D$, a circle in one strand with one closing mark. We have a cube with a single vertex $\choch(D_\emptyset)= \Z[x]\otimes\Lambda(\hat x)$ where $x_1=y_1=x$. The Hochschild differential is trivial while $\kappa'_\emptyset=p'(x)\hat x.$ Hence
 $$H_{v}H_{-}H_{+}(CH(D))=H_{-}(\choch(D_\emptyset))=\frac{\Z[x]}{(p'(x))}\hat x.$$
For $p(X)=X^{n+1}\in\Z[X]$, we get $\Z[x]/((n+1)x^n)$ in Hochschild degree 1. 

In the second easiest example, namely the word $1\in\Br(2)$, it turns out that
$\prescript{I}{\Z}E_1^{*,*}(D)=H_-H_+(KhR^\Z_p(D))$ already has nontrivial components in Hochschild degrees 1 and 2. This renders the basic idea used to prove the collapse of the rational $\prescript{I}{}E_*^{*,*}(D)$ unsuitable for the integral $\prescript{I}{\Z}E_*^{*,*}(D)$. In particular, it cannot be guaranteed that the second page $\prescript{I}{\Z}E_2^{*,*}(D)=H_vH_-H_+(KhR^\Z_p(D))$ is a link invariant. For this reason, an appropriate definition is:

\begin{definition}
\label{deflinkhomologiesint}
Let $L$ be a link presented by a braid diagram $D$, and let $p(X)=X^{n+1}\in\Z [X]$. The \emph{un-normalized integral Khovanov-Rozansky $\sln$ link homology} of $L$ defined by
$$KRH^\Z_n(L):=H_{v-}H_{+}(KhR^\Z_{p}(D)).$$
This is a link invariant after overall degree shifts, with value at the unknot given by
$$KRH^\Z_n(\unlink)=\frac{\Z[x]}{((n+1)x^{n+1})}.$$
The \emph{Krasner $\sln$ link homology} is the associated graded $\prescript{II}{\Z}E_\infty^{*,*}(D)$ from the Rasmussen spectral sequence. Both of these homologies coincide with the rational invariant $KRH_n(L)$ after rationalization.

\end{definition}

 Using the constructions from this chapter, the torsion $p'(x)$ we can impose on $\sln$ link homology of the unlink is constrained by the existence of the antiderivative $p(X)$ used in the formula for the class $\kappa'\in\choch(\coebs)$. Rationally, torsion $p'(x)=x^n$ could be directly obtained by using the potential $p(X)=\frac{1}{n+1}X^{n+1}$. To get torsion $p'(x)=x^n$ integrally we would need that the (a priori) rational class $\kappa'=\frac{1}{n+1}\sum_{j=1}^r \frac{x^{n+1}_j-y^{n+1}_j}{x_j-y_j}\hat x_j\in\choch(\coebs)\otimes\Q$ turned out to be integral. This is proved in the following section using equivariant cohomology, where this class has a special origin.

\section{Borel equivariant cohomology and $\sln$ link homology}\label{boreleqcohlink}

Algebraic $\sln$-link homology is computed through a complex having terms $\hoch(B_J)$ and differentials $d_-$ and $d_v$. In the topological setting, we already have an interpretation of $\hoch(B_J)$ as the equivariant cohomology $\htor^*(\BB_\tT(w_J))$, at least up to associated graded, with the $d_v$ differential induced by inclusions between different $\BB_\tT(w_J)$ or their suspensions. We now explore how $d_-$ differentials admit a geometric interpretation through the multiplication map $\BB_\tT(w_J)\to \Ur$. 

Let us review some facts about the equivariant cohomology of $\Ur$ from \cite{nitubroken2}. Denote $\coe=\htor^*(\Ur)$. We always think of $\Ur$  as a $\tT$-space by conjugation, so the set of $\tT$-fixed points in $\Ur$ is precisely the torus $\tT$. The map $\coe=\htor^*(\Ur)\to\htor^*(\tT)$ induced by inclusion $\iota:\tT\to\Ur$ is injective (see Appendix). We can then think of the ring $\coe$ as a subring of $\htor^*(\tT)=\HH^*(B\tT\times\tT)=\Z[x_1,\dots,x_r]\otimes \Lambda (\hat x_1,\dots,\hat x_r)$. We can also identify $\coe$ as Hochschild homology of $\coebs$.

Our main geometric input in this section is a map
$$ \lambda: \Ur_{h\tT}=E\Ur\times_\tT \Ur \longrightarrow \U := \colim \Ur.$$ 
Considering the concrete geometric model for $E\tT=E\Ur$ is as $r$-frames in $\C^\infty$, one defines $\lambda$ by taking an $r$-frame $F$ in $\C^\infty$ and a unitary matrix $g\in\Ur$ and sending the pair to the element in $\U$ that does the transformation prescribed by $g$ in the basis $F$, while doing the identity in the complement of $\mathrm{span}(F)$. To describe this map in cohomology, we first recall that $\HH^*(\U)$ is a Hopf algebra, exterior in the canonical primitive generators $\beta_0, \beta_1,\dots$ of degree $|\beta_n|=2n+1$.

\begin{proposition}(\cite{nitubroken2})
In cohomology, the map induced map by $\lambda$ is determined by further pullback along the maximal torus inclusion $\iota$,
$$\iota^*\lambda^*(\beta_n)=\sum_{i=1}^{r}x_i^n \hat{x_i}\in\htor^*(\tT).$$

\end{proposition}

Let us fix $n$. We will simplify our notation by writing $\beta_n$ for $\lambda^*(\beta_n)\in\coe$ and also for any further pullback $m^*\lambda^*(\beta_n)\in\htor^*(X)$ to a $\tT$-space over $\Ur$, $m:X\to \Ur$. Since $\beta_n\in\HH^*(\mathrm U)$ is itself an exterior generator, we know $\beta_n^2=0$. In consequence, (left) multiplication by $\beta_n$ is always a differential.

\begin{definition}
Let $m:X\to\Ur$ be a $\tT$-space over $\Ur$. Define
$$\nHH_\tT^*(X)=H(\htor^*(X),\beta_n),$$
homology with respect to the differential given by left multiplication by $\beta_n$. 
In the particular case $X=\BB_\tT(w_J)$ for a positive braid word $w_J$, denote
$$\nHH^*_\tT(w_J):=\nHH^*_\tT(\BB_\tT(w_J)).$$
\end{definition}

We think of $\nHH^*_\tT$ as a cohomological invariant of $\tT$-spaces over $\Ur$. We refer to $\nHH^*_\tT$ as an \emph{algebraically twisted} version of $\htor^*$. Note that the induced maps of our interest, namely those of the form $\iota_*$ and $\iota^*$, are $\coebs$-module (so also $\HH^*(\mathrm U)$-module) homomorphisms, thus inducing well defined maps at the level of $\nHH^*_\tT$.

In the simplest case $w_J=1\in\Br(r)$ we have $\htor^*(\BB_\tT(1))= \HH^*(B\tT\times\tT)=\Z[x_1,\dots,x_r]\otimes \Lambda (\hat x_1,\dots,\hat x_r)$ and $\beta_n=\sum_i x_i^n\hat x_i$. In consequence,
$$\nHH_\tT^*(\BB_\tT(1))=\frac{\Z[x_1,\dots,x_r]}{\left(\sum_i x_i^n\hat x_i\right)}\cdot\hat x_1\dots\hat x_r.$$

In terms of the KR-complex, these geometrically defined classes $\beta_n$ provide an alternative formulation of the $d_-$ differential. To see how this works,  let us consider a positive word $w_J$ and endow $\htor^*(\BB_\tT(w_J))$ with the structure of a filtered complex with respect to the Serre filtration (Section \ref{topmodelpositive}). Let $\widehat E_*^{*,*}(w_J)$ be the spectral sequence associated to this filtered complex $(\mathcal F^\bullet\htor^*(\BB_\tT(w_J)),\beta_n)$. Page zero is the associated graded $$\widehat E_0^{p,q}(w_J)=\assgr^p\left[\htor^{p+q}(\BB_\tT(w_J))\right]=\frac{\fil^p\htor^{p+q}(\BB_\tT(w_J))}{\fil^{p+1}\htor^{p+q}(\BB_\tT(w_J))},$$ which is also the $E_\infty$ page of the Serre spectral sequence. According to Proposition \ref{serressprop}, we have an isomorphism $\widehat E_0^{*,*}(w_J)\cong\hoch (B_J)$. We thus keep referring to the two gradings in this new spectral sequence $\widehat E_0^{*,*}$ as Soergel and Hochschild gradings, respectively. Denote the differential in $\widehat E_{s}^{*,*}(w_J)$ by $\widehat d_s$. With respect to our grading convention, $\widehat d_s$ has bidegree $(s,2n+1-s)$. We set out to prove that the only non-trivial differential in $\widehat E_*^{p,q}(w_J)$ is multiplication by a representative of $\beta_n$ in $\widehat E_{2n}^{2n,1}(w_J)=\widehat E_{0}^{2n,1}(w_J)\cong\hoch(B_J)$. 

\begin{lemma}
The class $\beta_n\in\coe$ lives in $\fil^{2n} \coe$. Thus, $\beta_n\in\fil^{2n}\htor^{*}(\BB_\tT(w_J))$.\end{lemma}

\begin{proof}
Consider the Serre spectral sequence $E_*^{*,*}$  for the fibration 
$$\tT\to \Ur_{h\tT}\to (\Ur/\tT)_{h\tT}.$$
As in the case of the fibration with total space $\BB_{h\tT}(w_J)$, we have $E^{*,*}_2=\choch(\coebs)$ and $E^{*,*}_3=E^{*,*}_\infty=\hoch(\coebs)$. Since the cohomology of the base, $\coebs$, is entirely concentrated in even degree, we know that $E^{2n+1,0}_\infty$ vanishes and so $E^{2n,1}_\infty=\fil^{2n}\htor^{2n+1}(\Ur)$. The entire spectral sequence consists of free abelian groups, so every page injects into its rationalization. Note that the cycle 
$$\kappa'=\sum_{i=1}^r \frac{1}{n+1}\frac{x_i^{n+1}-y_i^{n+1}}{x_i-y_i}\hat x_i\in E_2^{2n,1}\otimes\Q.$$
represents $\beta_n$ because, at the infinite page, the map induced by fixed points inclusion $\iota: \tT\to \Ur$ sends $\kappa'\mapsto \sum_i x_i^n \hat x_i$. By injectivity we conclude that $\beta_n=\kappa'\in\fil^{2n}\htor^{2n+1}(\Ur)$. The second part of the proposition follows by pulling back along the multiplication map, which respects the relevant fibrations.
\end{proof}

Since the filtration is multiplicative, the following is a direct consequence of the lemma.

\begin{corollary}
The differentials $\widehat d_s$ in $\widehat E_*^{*,*}(w_J)$ are all zero for $s<2n$. At the page $\widehat E_{2n}^{*,*}(w_J)$ the differential is multiplication by $\beta_n\in \widehat E^{2n,1}_{2n}(w_J)$.
\end{corollary}

A priori, the expression defining $\kappa'\in \choch(\coebs)\otimes\Q$  is rational as it involves the denominators in the potential $p(X)=\frac{1}{n+1}X^{n+1}$. In Hochschild homology, however, the class $\kappa'\in \hoch(\coebs)\otimes \Q$ is integral, being the same class of $\beta_n$. We provide explicit integral representatives $\bet_n\in\choch(\coebs)$ for $n\leq 2$:
\begin{align*}
\bet_0&=\sum_{k=1}^r \hat x_k,\\
\bet_1&=\sum_{k=1}^r (x_k-S_{k-1})\hat x_k,\\
\bet_2&=\sum_{k=1}^r (x^2_k+x_{k}(S_{k}-S_{k-1})-S_kS_{k-1})\hat x_k.\\
\end{align*}
The fact that these are Hochschild cycles is a matter of Schubert calculus. We refer the reader to Proposition \ref{schubertappendix} in the Appendix for the relevant formulas in $\coebs$ from which the cycle condition for $\bet_0$, $\bet_1$ and $\bet_2$ follow directly. It seems plausible, if laborious, to obtain explicit or recursive formulas for higher $n$.

Now we turn our attention to link homology. Let $D$ be the link diagram given by closure of a braid word $w_I$ and let 
$$\KhRn(D)=\bigoplus_{J\in\II}\choch(B_J)$$ be, at each vertex, the Hochschild complex of the associated Soergel bimodule $B_J$. We think of $\KhRn(D)$ as a triple complex with the usual $d_+$ and $d_v$, while $d_-$ is given by multiplication by (the pullback of) $\bet_n$ at each vertex. This should be contrasted with the complex $KhR_p^\Z(D)$ in the previous chapter. Both are the same as triply graded $\htor^*$-algebras. Nonetheless, if $p(X)=X^{n+1}\in\Z[X]$, we have $$\left(H_+(KhR_p^\Z(D)),d_-,d_+\right)=\left(H_+(\KhRn(D)),(n+1)\cdot d_-,d_+\right),$$ suggesting that this definition is correcting the torsion that appears in the formulation with potentials.

\begin{definition}
Given a braid diagram $D$ for a link $L$, we define the \emph{normalized integral Khovanov-Rozansky $\sln$ link homology} as the iterated homology
$$\KRHn(L):=H_vH_+H_-(\KhRn(D)).$$
\end{definition}

We now turn our attention into the verification of certain properties that this normalized integral theory has in common with the rational theory $KRH_n$. According to our discussion in the previous chapter, some of these properties do not hold for the un-normalized counterpart $KRH^\Z_n$.

\begin{theorem}
There is an isomorphism
$$\KRHn(L)=H_vH_-H_+(\KhRn(D))\cong H_{v-}H_+(\KhRn(D)).$$
Therefore $\KRHn(L)$ is (up to overall shifts) an integral invariant of $L$ with value at the unknot
$$\KRHn(\unlink)=\frac{\Z[x]}{(x^n)},$$
which is free over $\Z$. The link homology $\KRHn$ coincides with the usual Khovanov-Rozansky $\sln$ link homology $KRH_n$ after rationalization.
\end{theorem}

\begin{proof}

As usual, the double complex comes with two spectral sequences
\begin{align*}\prescript{I}{}E_1^{*,*}=H_-H_+(\KhRn(D))\Longrightarrow H_{v-}H_+(\KhRn(D)),\\
\prescript{II}{}E_1^{*,*}=H_vH_+(\KhRn(D))\Longrightarrow H_{v-}H_+(\KhRn(D)),
\end{align*}
graded by Hochschild-cubical bigrading. On the one hand, the Rasmussen spectral sequence $\prescript{II}{}E_*^{*,*}$ ensures that the total homology $H_{v-}H_+(\KhRn(D))$ is a link invariant from the fact that $\prescript{II}{}E_1^{*,*}$ is already integral HOMFLY-PT homology. Moreover, after rationalization this coincides with the rational Rasmussen spectral sequence, identifying $H_{v-}H_+(\KhRn(D))\otimes\Q$ with the usual Khovanov-Rozansky $\sln$ link homology.

 The isomorphism $H_{v-}H_+(\KhRn(D))\cong H_vH_-H_+(\KhRn(D))$ is a consequence of the fact that the other spectral sequence $\prescript{I}{}E_*^{*,*}$ collapses at the second page into a single vertical $\prescript{I}{}E_2^{r,*}$, just like in the rational case. This follows from Lemma \ref{tophochdeg} by noticing that $\prescript{I}{}E_1^{*,*}=\bigoplus_J H(\hoch(B_J),\beta_n)$ already lives entirely in Hochschild degree $r$. That is, the first page is already concentrated in a single horizontal line and so $d_1$ is the last possibly non-trivial differential.
\end{proof}

\subsection{Properties of $\nHH^*_\tT$}
 The next lemma, which is the technical core of this section, concludes the computation of the spectral sequence $\widehat E_{*}^{*,*}(w_J)$.  We have collapse into a single horizontal line at the page $2n+1$, giving us a canonical identification of the homology $\nHH_\tT^*(w_J)$ in terms of Hochschild homology of Soergel bimodules.

\begin{lemma}\label{tophochdeg}
\begin{enumerate}[label=(\alph*)]
\item Hochschild homology $\hoch(B_J)$ is a free $\htor^*$-module.
\item The homology $H(\hoch(B_J),\beta_n)$ is concentrated in top Hochschild degree $r$. 

\item The homology $H(\hoch(B_J),\beta_n)$ is free over $\Z$ and the $\Z[x_1,\dots,x_r]$-module structure factors through  $\Z[x_1,\dots,x_r]/(x_1^n,\dots,x_r^n)$.
\item The spectral sequence $\widehat E_{*}^{*,*}(w_J)$ collapses into a single line at the page $2n+1$, with $$\widehat E_{\infty}^{*,*}(w_J)=\widehat E_{2n+1}^{*,*}(w_J)=H(\hoch(B_J),\beta_n),$$ which is, in turn, canonically isomorphic to $\nHH_\tT^*(w_J)(w_J)$.
\end{enumerate}
\end{lemma}

\begin{proof}
Let $J=(i_i,\dots, i_l)$. We prove (a), (b) and (c) by induction on $N(J)=r+\sum_j i_j$, quantity know as the \emph{complexity} of $w_J\in\mathrm{Br}(r)$. For $N=1$ we necessarily have $r=1$ and $J=\phi$, so that $\hoch(B_\phi)=\Z[ x ]\otimes\Lambda(\hat x)$ and $\beta_n=x^n\hat x$. The statements are clearly true in this case. For $N>1$ there exist four possibilities: 
\begin{enumerate}[label=\emph{(\roman*)}]
\item the index $r-1$ does not appear in $J$, 
\item the index $r-1$ appears exactly once in $J$, 
\item there are two repeated consecutive indices $i_j=i_{j+1}$ in $J$ or 
\item there are three consecutive indices satisfying $i_{j}=i_{j+1}+1=i_{j+2}$ in $J$.
\end{enumerate}
 We do induction by reducing to sequences of lower complexity.

We shall recall a few useful facts from \cite{nitubroken1} and \cite{nitubroken2}. First of all, we can always permute indices in $J$ cyclically without changing $\BB_{\tT}(w_J)$ up to $\tT$-equivariant homeomorphism. For this reason, in all cases we assume that the situation at hand is happening in the last indices of the sequence $J$. Next, for case (i) and (ii) we use the following $\tT$-equivariant isomorphisms:
\begin{align}
\label{fibr1}\BB_{\tT}(w_J)&\cong \BB_{\tT^{(r-1)}}\left(w^{(r-1)}_{I}\right)\times \mathrm{SU}(2),\\
\label{fibr2}\BB_{\tT}(w_{I})&\cong \BB_{\tT^{(r-1)}}\left(w^{(r-1)}_{I}\right)\times \Delta.
\end{align}
The notation here is the following: in $J=(i_1,\dots,i_l)$ the last index $i_l=r-1$ is the only index with that value, so that $I=(i_1,\dots,i_{l-1})$ has no occurrence of $r-1$. The word $w^{(r-1)}_{I}$ lives in a braid group with one less strand $\mathrm{Br}(r-1)$. The circle $\Delta$ is the maximal torus in $\mathrm{SU}(2)$, both thought of as groups of $r\times r$ matrices differing from the identity only in the last 2 by 2 block. The subtorus $\tT^{(r-1)}\leq \tT$ consists of diagonal matrices in which the first $r-1$ entries multiply to 1.  The torus $\tT$ acts by conjugation on all terms, with the last entry acting trivially on $\BB_{\tT^{(r-1)}}\left(w^{(r-1)}_{I}\right)$ and the first $r-2$ entries acting trivially on $\mathrm{SU}(2)$ and $\Delta$.

\emph{Case (i)}. We deal with a sequence $I$ where the index $r-1$ does not appear, so we might consider the word $w_I^{(r-1)}$ which is the same $w_I$ thought of as living in $\Br(r-1)$, thus of lower complexity. The identity (\ref{fibr2}) induces a pullback diagram
$$\begin{tikzcd}\BB_{h\tT}(w_I) \ar[r] \ar[d] &\Delta_{h\tT}\ar[d] \\ \BB_{\tT^{(r-1)}}\left(w^{(r-1)}_{I}\right)_{h\tT}\ar[r] & B\tT.
\end{tikzcd}
$$
whence one gets a canonical induced map
\begin{equation}\label{decomp}\htor^*\left(\BB_{\tT^{(r-1)}}\left(w^{(r-1)}_{I}\right)\right)\otimes_{\htor^*} \htor^*(\Delta) \to\htor^*(\BB_\tT(w_I)).\end{equation}
that turns out to be an isomorphism by collapse of the Eilenberg-Moore spectral sequence. The map $\BB_{h\tT}(w_J)\to\mathrm U$ factors as $$\BB_{h\tT}(w_J)\to\BB_{\tT^{(r-1)}}\left(w^{(r-1)}_J\right)_{h\tT}\times \Delta_{h\tT}\to \mathrm U\times\mathrm U\to\mathrm U,$$
where the last map is multiplication. In cohomology, we use this factorization and the fact that $\beta_n\in\HH^*(\mathrm U)$ is primitive to guarantee that the isomorphism (\ref{decomp}) indeed corresponds $\beta_n\otimes 1+1\otimes \beta_n$ to $\beta_n$, being thus enhanced into an isomorphism of chain complexes. At the level of associated graded, and after cancelling out a spare variable, we have an isomorphism 
\begin{equation}\label{decompHH}\hoch\left(B^{(r-1)}_{I}\right)\otimes_{\HH_{\tT^{(r-1)}}^*} (\htor^*\otimes\Lambda(\hat x))\to\hoch(B_I).\end{equation}
Consequently, one could think of $\hoch(B_I)$ as a double complex with horizontal (resp. vertical) differential $\beta_n\otimes 1$ (resp. $1\otimes\beta_n$) and horizontal grading by Hochschild degree in $\hoch(B^{(r-1)}_{I})$ (resp. Koszul degree in $\htor^*\otimes\Lambda(\hat x)$).  As such, it comes with a double complex spectral sequence $E^{*,*}_*$ associated to the horizontal filtration. This should not be confused with the spectral sequence $\widehat E_{*}^{*,*}$ that we ultimately want to compute. The factor $\htor^*\otimes\Lambda(\hat x)$ is free over $\htor^*$ and the $d_-$-differential there is given by $\beta_n=(x_{r-1}^n-x_r^n)\hat x$. The homology $H(\htor^*\otimes\Lambda(\hat x),\beta_n)=\frac{\htor^*}{(x_{r-1}^n-x_r^n)}\hat x$ is still free over $\HH_{\tT^{(r-1)}}^*$. Such freeness suffices to guarantee that the second page $E^{*,*}_2$ (i.e. the iterated homology) has the form
\begin{equation}\label{decompHn}E^{*,*}_2=H\left(\hoch(B^{(r-1)}_I),\beta_n\right)\otimes_{\HH_{\tT^{(r-1)}}^*}\frac{\htor^*}{(x_{r-1}^n-x_r^n)}\hat x.\end{equation}
Note that $H(\hoch(B^{(r-1)}_I),\beta_n)$ lives in Hochschild degree $r-1$ by inductive hypothesis on the lower complexity word $w^{(r-1)}_I$, while $\hat x$ is in degree 1. We see now that this spectral sequence has already collapsed into a single spot $E^{r-1,1}_2$ at the second page. This term is therefore isomorphic to $H(\hoch(B_I),\beta_n)$ and the total Hochschild degree is $r$ as desired. Regarding property (b), notice that the isomorphism (\ref{decompHH}) already presents $\hoch(B_I)$ as a free $\htor^*$-algebra. 

For property (c) note from (\ref{decompHn}) that, as $\Z[x_1,\dots,x_r]$-modules,
$$H(\hoch(B_I),\beta_n)\cong\frac{M[x_r]}{(x_{r-1}^n-x_r^n)M[x_r]},$$
where $M=H\left(\hoch(B^{(r-1)}_I),\beta_n\right)$ has trivial action of $x_r$. By inductive hypothesis we know $x_{r-1}^nM=0$. Thus $$H(\hoch(B_I),\beta_n)\cong\frac{M[x_r]}{x_r^nM[x_r]}.$$ In view of this, the properties in (c) follow for $H(\hoch(B_I),\beta_n)$.

\emph{Case (ii)}. We start with $J$ in which the index $r-1$ appears only once and in the last position. We want to look at $\BB(w_J)$ in terms of the sequence $I$ in which the last index $i_l=r-1$ is removed. Completely analogous arguments to part (i) give us an isomorphism $\htor^*(\BB(w_J))= \htor^*(\BB(w^{(r-1)}_{I}))\otimes_{\htor^*}\htor^*(\mathrm{SU}(2))$ and then at the level of Hochschild homology $\hoch(B^{(r-1)}_{I})\otimes_{\HH_{\tT^{(r-1)}}^*} (\htor^*\otimes\Lambda(\oldgamma))=\hoch(B_I)$. In this case  $(\htor^*\otimes\Lambda(\oldgamma))$ has $d_-$ differential $\beta_n=\frac{x^n_{r-1}-x_r^n}{x_{r-1}-x_r}\oldgamma$. The conclusions follow similarly.

\emph{Case (iii)}. Take $J=(i_1,\dots,i_l)$ with $i_{l-1}=i_{l}=s$, and let $J'$ be the subsequence in which $i_{l}$ omitted. The multiplication map $G_s\times_{\tT}G_s \to G_s$ induces a map at the level of Soergel bimodules (and in fact, at the level of $\coebs$-algebras) $B_{J'}\to B_J$. This map can be expressed in terms of the formula (\ref{deltaform}), where generators are respectively $\delta'_k\in B_{J'}$ and $\delta_k\in B_{J}$, by
$$\delta'_{k} \mapsto \left\lbrace \begin{matrix}
\delta_k & \mbox{for }k<l-1,\\
\delta_{l-1}+\delta_{l} & \mbox{for }k=l-1. 
\end{matrix}\right. $$
One verifies that $B_J$ is a free $B_{J'}$-module with basis $\lbrace 1,\delta_{l}\rbrace$. This means that
$$B_J\cong B_{J'}\oplus B_{J'}\{2\}$$
as modules over $\coebs$, where curly brackets represent a shift in cohomological degree. Taking Hochschild homology we have an isomorphism
$$\hoch(B_J)\cong\hoch(B_{J'})\oplus \hoch(B_{J'})\lbrace 2\rbrace$$
as $\hoch(\coebs )$-modules. Notice that there is no shift in Hochschild degree. The induction hypotheses applies to each of the summands.

\emph{Case (iv)}. Take $J=(i_1,\dots,i_l)$ with $i_{l-2}=i_{l}=s$ and $i_{l-1}=s-1$. We have a multiplication map $G_{s}\times_{\tT}G_{s-1}\times_{\tT}G_{s}\longrightarrow\tT^{s-2}\times  \mathrm U (3)\times \tT^{r-s-1}$ which induces the inclusion of the subalgebra $A\subseteq B_J$ generated by $A=\langle\delta_1,\dots,\delta_{l-3},\delta_{l-2}+\delta_{l},\delta_{l-1}\rangle$. In order to obtain a useful decomposition, we want a complement of $A$ inside $B_J$. Consider the class 
$$\lambda:= \delta_{l-1} + (x_{s-1}-x_{s})+ 2 \textstyle\sum_{k<l-1,\, i_k=s-1} \delta_k -\textstyle\sum_{k<l-1,\, |s-1-i_k|=1} \delta_k $$
Note that the map $A\to\lambda A$ which multiplies by $\lambda$ has kernel $\delta_{l-1} A$, so $\lambda A\cong A/(\delta_{l-1})\{2\}$ as an $A$-module. But we have an alternate description of this: if $J'$ is the subsequence of $J$ with the indices $i_{l-1},i_{l}$ omitted, then the map $A\to B_{J'}$ induced from inclusion $G_s\to \mathrm{U}(3)$ is surjective and also has kernel $\delta_{l-1} A$. This way we have an isomorphism $\lambda A\cong B_{J'}\{2\}$ as $A$-modules, and therefore a decomposition
\begin{equation}\label{splittingU3}B_J\cong A\oplus B_{J'}\{2\}.\end{equation}
This is also a decomposition of $\coebs$-modules by further composing with the map $\coebs\to A$ induced by multiplication. One can follow the same argument switching the roles of $s$ and $s-1$, letting $\widetilde J$ be $J$ except that $i_{l-2}=i_{l}=s-1$ and $i_{l-1}=s$, while $\widetilde J '$ is $\widetilde J$ with the indices $i_{l-1}$ and $i_{l}$ removed. In this case there is an isomorphism of $\coebs$-modules
$$B_{\widetilde{J}}\cong A\oplus B_{\widetilde J'}\{2\},$$
and so
$$B_J\oplus B_{\widetilde J'}\{2\} \cong B_{\widetilde J}\oplus B_{J'}\{2\}.$$
We then run Hochschild homology and $\beta_n$-homology and the result follows by observing that $J', \widetilde J$ and $\widetilde J '$ all have lower complexity than $J$.
This concludes the induction, so $\widehat{E}_1^{*,*}(w_J)$ lives in top Hochschild degree. Properties (b) and (c) also follow.

Regarding the spectral sequence $\widehat{E}_*^{*,*}(w_J)$, we have shown that the page $2n+1$ collapses into the single horizontal line $\widehat E^{*,r}_{2n+1}(w_J)$. There can be no further non-zero differentials, and so we have the canonical identifications
$$H(\hoch(B_J),\beta_n)=\widehat E^{*,*}_{2n+1}(w_J)=\widehat E^{*,*}_{\infty}(w_J)= H(\htor^*(\BB_\tT(w_J)),\beta_n)=\nHH_\tT^*(w_J).$$\end{proof}

\begin{remark}Observe that the $\Z[x_1,\dots,x_r]/(x_1^n,\dots,x_r^n)$ is, up to shifts, the homology of the unlink in $r$-components. The $\Z[x_1,\dots,x_r]/(x_1^n,\dots,x_r^n)$-module in the Lemma above is compatible with edge maps and thus it yields a module structure on $\KRHn(L)$ when the link $L$ is presented by a braid diagram $D$ in $n$ strands. This module structure is not a link invariant on the nose, but one can expect it to be after identifying variables that correspond to the same component of the link.\end{remark}

The link invariance of $\KRHn$ as a $\Z$-module, albeit directly established through the Rasmussen spectral sequence, admits an alternative verification by checking INS conditions, as sketched at the end of our first chapter. We close this section with the verification of such conditions for $\nHH_\tT^*$.

\begin{corollary}\label{twistedisisn}
\begin{enumerate}[label=(\alph*)]
\item\label{a1} For a sequence $J_1=(i_1,\dots,i_{l-1})$ not containing either $\pm(r-1)$ and $J_2=(i_1,\dots,i_{l-1},r-1)$, the inclusion $\iota:\BB_\tT(w_{J_1})\to\BB_\tT(w_{J_2})$ induces an injective positive edge map $\iota^*$ in $\nHH_\tT^*$.\\

\item\label{b1} For a sequence $J_1=(i_1,\dots,i_{l-1})$ not containing either $\pm(r-1)$ and $J_2=(i_1,\dots,i_{l-1},-(r-1))$, the inclusion $\iota:\BB_\tT(w_{J_1})\to\BB_\tT(w_{J_2})$ induces a surjective negative edge map $\iota_*$ in $\nHH_\tT^*$.\\

\item\label{c1} For positive $s$ and a sequence $J=(i_1,\dots,i_{l-3},s,s-1,s)$ or $J=(i_1,\dots,i_{l-3},s-1,s,s-1)$, let $\mathcal{X}=G_{i_1}\times_{\tT}\cdots\times_{\tT}G_{i_{l-3}}\times_{\tT}  (\tT^{s-2}\times  \mathrm U (3)\times \tT^{r-s-1})$. Then the map $\BB_\tT(w_J)\to \mathcal{X}$ induced by multiplication is an injective homomorphism in $\nHH_\tT^*$.\end{enumerate}

\end{corollary}

\begin{proof} This is a corollary of the proof of Lemma \ref{tophochdeg}, so we rely on the details thereof. For parts \ref{a1} and \ref{b1}, we have the two maps $\iota^*:\htor^*(\mathrm{SU(2)})\to \htor^*(\mathrm{\Delta})$ and $\iota_*:\htor^*(\mathrm{\Delta})\{2\}\to \htor^*(\mathrm{SU(2)})$ induced by inclusion $\iota:\Delta\to\mathrm{SU(2)}$. At the level of $\twE_{\infty}^{*,*}$ we have the decompositions 
\begin{align*}H(\hoch(B_{J_1}),\beta_n)&=H(\hoch(B^{(r-1)}_{J_1}),\beta_n)\otimes_{\HH_{\tT^{(r-1)}}^*} H(\htor^*\otimes\Lambda(\hat x),(x_{r-1}^n-x_r^n)\hat x),\\ H(\hoch(B_{J_2}),\beta_n)&=H(\hoch(B^{(r-1)}_{J_1}),\beta_n)\otimes_{\HH_{\tT^{(r-1)}}^*} H\left(\htor^*\otimes\Lambda(\oldgamma),\frac{x_{r-1}^n-x^n_r}{x_{r-1}-x_r}\gamma\right).\end{align*} The positive (resp. negative) edge map induces the identity on $H(\hoch(B^{(r-1)}_I),\beta_n)$ tensored with the left to right (resp. right to left) map in $$\begin{tikzcd}[row sep=10]H\left(\htor^*\otimes\Lambda(\oldgamma),\frac{x_{r-1}^n-x^n_r}{x_{r-1}-x_r}\oldgamma\right)\ar[d,equal]&H(\htor^*\otimes\Lambda(\hat x),(x_{r-1}^n-x_r^n)\hat x)\ar[d,equal]\\ \oldgamma\cdot\htor^* \left/{\left(\frac{x^n_{r-1}-x_r^n}{x_{r-1}-x_r}\right)}\right.\ar[r,shift left=.5ex,"\iota^*"] & \hat x\cdot{\htor^*}/{(x_{r-1}^n-x_r^n)}\ar[l,shift left=.5ex,"\iota_*"],
\end{tikzcd}$$ which are determined by $\iota^*(\oldgamma)=(x_{r-1}-x_r)\hat x$ and $\iota_*(\hat x)=\oldgamma$. The map $\iota^*$ is injective, while $\iota_*$ is surjective, and it is easy to verify that this remains true after tensoring with $H(\hoch(B_I^{(r-1)}),\beta_n)$ over $\HH^*_{\tT^{(r-1)}}$. This gives the desired INS properties on $\twE_{\infty}^{*,*}$ and thus on $\nHH_\tT^*$.

For part \ref{c1}, the splitting in (\ref{splittingU3}) induces a splitting of $\hoch(\coebs)$-modules
$$\hoch(B_J)= \hoch(A)\oplus \hoch(B_{J'})\{2\}.$$
Since the differentials in $\twE^{*,*}_*$ respect this splitting, the map $\BB_{\tT}(w_J)\to\mathcal X$ still gives a splitting at $\twE^{*,*}_\infty$, hence at $\nHH_\tT^*(w_J)$.
\end{proof}

\section{Equivariant $\sln$ link homology and twisted cohomology}\label{twisted}

In this chapter we consider what can be called an integral universal equivariant $\sln$-link homology over the extended coefficient ring $\R=\Z [b_1,b_2,\dots]$, in which $b_i$ has cohomological degree $-2i$. Up to shifts, the homology of the unknot $\unlink$ will be
$$\KRHu(\unlink)=\frac{\R [\![x]\!]}{(b_1x+b_2x^2+b_3x^3+\cdots)}.$$

One way to do this is to follow the recipe from the previous section by enlarging the base ring: instead of $\htor^*(X)$ we work with extended coefficients $\hcal_\tT^*(X)=\htor^*(X)\ctp\R$ and then we get a $d_-$ differential by pulling back the `universal' class
$$\beta_u=\sum_{i=1}^\infty \beta_i b_i \in \hcal^*(\mathrm U).$$ 
Taking 
$$\uHH_\tT^*(X)=H(\hcal_\tT^*(X),\beta_u)$$
for any $\tT$-space $X$ over $\Ur$, and in particular
$$\uHH_\tT^*(w_J)=H(\hcal_\tT^*(\BB_\tT(w_J)),\beta_u),$$
we get a theory which is formally similar to $\nHH_\tT^*$. We should refer to $\uHH_\tT^*$ as an algebraically twisted version of $\hcal_\tT^*$. In fact, the integral version $\nHH_\tT^*(w_J)$ is easily recovered through the following specialization. Let $\R\to \Z$ be the evaluation that sets $b_n\mapsto 1$ and $b_i\mapsto 0$ for $i\neq n$. Then 
$$(\HH_\tT^*(\BB_\tT(w_J))\ctp\R,\beta_u)\otimes_\R\Z=(\HH_\tT^*(\BB_\tT(w_J)),\beta_n),$$
which is precisely the chain complex computing $\nHH_\tT^*(w_J)$. Notice that the total cohomological degree gets raised by $2n$ since $|b_n|=-2n$. 

The results about $\nHH_\tT^*$ from the previous section apply with minor adjustments to $\uHH_\tT^*$ as well.

\begin{lemma}\label{tophochdeguniv} Given a positive braid word $w_J\in\Br(r)$, the complex $(\htor^*(\BB_\tT(w_J))\ctp\R,\beta_u)$ can be equipped with the Serre filtration to obtain a spectral sequence of signature
$$\widehat E_1^{*,*}=\hoch(B_J)\ctp\R\Longrightarrow\uHH_\tT^*(w_J)$$
with differential $d_1$ given by the action of $\beta_u$.
\begin{enumerate}[label=(\alph*)]
\item The homology $\widehat E_2^{*,*}=H(\hoch(B_J)\ctp\R,\beta_u)$ is concentrated in top Hochschild degree $r$. It is torsion free over $\R$ and the $\R[\![x_1,\dots,x_r]\!]$-module structure factors through $\R[\![x_1,\dots,x_r]\!]/(p(x_1),\dots,p(x_r))$, where $p(x)=\sum_{i>0}b_ix^i$.
\item\label{b2} The spectral sequence $\widehat E^{*,*}_*$ collapses into a single line at the second page, giving canonical isomorphisms
$$\uHH_\tT^*(w_J)\cong \widehat E_2^{*,r}=H(\hoch(B_J)\ctp\R,\beta_u).$$
In particular, $\uHH_\tT^*(w_J)$ lives only in total cohomological degrees with the same parity as $r$.
\item\label{c2} The functor $\uHH_\tT^*$ satisfies the INS conditions. More specifically, the three maps indicated in Corollary \ref{twistedisisn} also induce injections (surjection in the case of $\iota_*$) in $\uHH_\tT^*$.
\end{enumerate}
\end{lemma}

Given then a link $L$ represented by a braid word $w_I$, we can apply $\uHH_\tT^*$ to the cube $ \sS_J(w_\bullet):\II\to \mathrm{Spaces}^\tT_* $ and obtain a cubical complex $\bigoplus_{J\in\II}\uHH_\tT^*$ with cubical differential $d_v$ defined in the usual way.

\begin{theorem}
Let $L$ be a link presented by a braid word $w_I$ and $\R=\Z[b_1,b_2,\dots]$. The homology
$$\KRHu(L)=H_v\left(\bigoplus_{J\in\II}\uHH_\tT^*(w_J)\right)=H_v \left(\bigoplus_{J\in\II}H(\hoch(B_J)\ctp\R,\beta_u)\right)$$
is a link invariant up to overall degree shifts. Its value at the unlink is given by
$$\KRHu(\unlink)=\frac{\Z[b_1,b_2,\dots] [\![x]\!]}{(b_1x+b_2x^2+b_3x^3+\cdots)}.$$
\end{theorem}

The functor $\uHH_\tT^*$ and the link homology $\KRHu$ first appeared in \cite{nitubroken2}, in relation to a twisted equivariant cohomology theory that we will denote here by $\thetah^*_\tT$. The algebraically twisted cohomology $\uHH_\tT^*$ should be though of as an approximation to the topologically twisted $\thetah^*_\tT$, in the sense that there exists a spectral sequence from the former to the latter. The remainder of this chapter is devoted to explain this framework.

Generally speaking, the cohomology $\thetah^*_\tT$ is also a functor of $\tT$-spaces over $\Ur$, such as $\BB_\tT(w_J)$. It is a \emph{topologically twisted} version of the cohomology theory associated to the spectrum $\hcal=\HH\Z\wedge B\mathrm U_+$. Note that this spectrum satisfies $\pi_*(\hcal)=\HH_*(B\mathrm U)=\R$, and so $\hcal^*(-)=\HH^*(-)\ctp\R$ as the notation suggests. For completeness, we give a presentation of the formalism of topological twists in Section \ref{deftwisted}, then in Section \ref{untwtotw} we leap directly into the spectral sequence that relates $\uHH_\tT^*$ and $\thetah^*_\tT$.

\subsection{A definition of $
\thetah_{\tT}^*$} \label{deftwisted}
We proceed to give a definition using the conventions for twisted cohomology in \cite{satiwester}. This language may also come handy when exploring other known twisted cohomology theories of the spaces $\BB_\tT(w_J)$ in future work.

Let us start by reviewing the notion of twisted cohomology we will be working with. Recall that generalized cohomology theories are represented by spectra $\EE$, in the sense that the functor $\EE^\bullet(X)=\pi_0 F(X,\Sigma^\bullet\EE)$ satisfies the usual Eilenberg-Steenrod axioms. Here $F(-,-)$ denotes the mapping spectrum.

Let $\EE$ be a connective $A_\infty$-ring spectrum, and let $\GL(\EE)\subseteq\Omega^\infty \EE$ consist of those components which are invertible in $\pi_0(\EE)$. There is an associated principal fibration
$$\GL(\EE)\to E\GL(\EE) \to B\GL(\EE).$$ A \emph{(topological) twist} of $\EE$ by $Z$ is a map $\theta:Z\to B\GL(\EE)$. Given $f:X\to Z$, we have a principal $\GL(\EE)$-bundle $P_f\to X$ by pulling back the universal bundle along $\theta\circ f:X\to B\GL(\EE)$. The \emph{generalized Thom spectrum}, defined as $X^f := \Sigma_+^\infty P_f \wedge_{\Sigma_+^\infty\GL(\EE)}\EE$ promotes $X$ into an $\EE$-module by using the twist. We define the \emph{twisted $\EE$-cohomology of $X$ with twist $\theta$} as
$$\prescript{\theta}{}\EE^k(X;f)=\pi_{0}F_{\EE}\left(X^f,\Sigma^k\EE\right),$$
where $F_{\EE}(-,-)$ is now the $\EE$-module mapping spectrum. Homotopic choices of $f$ produce isomorphic cohomologies $\prescript{\theta}{}\EE^k(X;f)$, and we will often omit $f$ from the notation.

In the situation at hand, we consider the space $B\mathrm U$ with the $A_\infty$-space structure $\oplus$ classifying the Whitney sum of complex vector bundles. From this, we get an $A_\infty$-ring spectrum structure on $\hcal=\HH\Z\wedge \bup$. To get a twist for $\hcal$, let $B\mathrm U\simeq\Omega^\infty(\mathbb S\wedge B\mathrm U_+)\to \GL(\hcal)\subseteq \Omega^\infty(\HH\Z\wedge B\mathrm U_+)$ be the map induced from the unit $\mathbb{S}\to\HH\Z$. Composing with the projection $\Z\times B\U\to B\U$, we have a map of $A_\infty$-spaces $\Z\times B\U \to \GL(\hcal)$, which can be delooped into a map $$B_{\oplus}(\Z\times B\mathrm U)\to B\GL(\hcal).$$ Now, Bott periodicity gives an equivalence $\mathrm{U}\simeq B_{\oplus}(\Z\times B\mathrm U)$. Thus we obtain the desired twist $\theta: \mathrm U\to B\GL(\hcal)$.

Given a $\tT$-space $m:X\to\Ur$ we then use the sequence of maps
$$\begin{tikzcd}[column sep=8mm] X_{h\tT}\ar{r}{m_{h\tT}}& \Ur_{h\tT}\ar{r}{\lambda}&\mathrm U \ar{r}{\theta}& B\GL(\hcal).\end{tikzcd}$$
to define twisted equivariant cohomology.

\begin{definition}
Let $m:X\to\Ur$ be a $\tT$-space over $\Ur$. Define
$$\thetah^*_\tT(X)=\thetah^*(X_{h\tT};\lambda \circ m_{h\tT}).$$
\end{definition}
\begin{remark} 
Let us compare this definition to the one presented in the source \cite{nitubroken2}. In the latter, one begins with the spectrum $\hcal$ and gets a theory $\prescript{\theta}{}{\mathscr{H}}^*_{\Ur}$ for $\Ur$-spaces over $\Ur$ (acting on itself by conjugation). Given a $\tT$-space $X$ over $\Ur$, there is an induced $\Ur$-space over $\Ur$ by composing
$$\Ur^{\mathsf r}\times_{\tT} X\to \Ur^{\mathsf r}\times_{\tT} \Ur\to\Ur,$$
the second map sending $[g,h]\mapsto ghg^{-1}$. The superscript in $\Ur^{\mathsf r}$ indicates $\tT$-action by right multiplication as opposed to our usual conjugation action.  Then the precise identification is
$$\thetah_{\tT}^*(X)\cong \prescript{\theta}{}{\mathscr{H}}^*_{\Ur}(\Ur^{\mathsf r}\times_{\tT} X).$$

\end{remark}

\subsection{The untwisted-to-twisted spectral sequence and link homology}\label{untwtotw}
Given a $\tT$-space $X$, one has a $\hcal^*_{\tT}(X)$-module $\prescript{\theta}{}{\hcal}^*_{\tT}(X)$. Additionaly, this comes with an `untwisted-to-twisted' spectral sequence with differentials determined by multiplication by an odd class. This is precisely how an algebraically twisted cohomology like $\uHH_\tT^*$ appears in this setting. 

The spectral sequence can be parsed using the language of the previous section as follows. First notice that, besides the cohomological degree, there is a $b$-\emph{degree} in the ring $\R=\Z[b_1,b_2,\dots]$. This is the polynomial grading in which all $b_i$ are considered of degree 1.
The ring spectrum $\hcal$ admits a (co)-filtration by $\hcal$-modules 
\begin{equation}\label{bufiltration}\hcal_0=\hcal \leftarrow\hcal_1\leftarrow\hcal_2\leftarrow\cdots\end{equation} which lifts the decreasing filtration by $b$-degree on $\R$. This induces, in turn, a filtration $\left\{\pi_0F_\hcal\left(X^f,\Sigma^k\hcal_\bullet\right)\right\}$. The following spectral sequence is the one associated to this filtration.

\begin{proposition} \emph{(\cite{nitubroken2}, Sec. 4)}
Let $X$ be a $\tT$-space over $\Ur$. There is a spectral sequence induced by the filtration (\ref{bufiltration}) with signature
\begin{equation}\label{nitu}\thetae_1^{*,*}(X)=\hcal^*_{\tT}(X)=\htor^*(X)\ctp \R \Longrightarrow \thetah^{*}_\tT(X),\end{equation} where the bigrading is such that $\thetae_1^{p,q}(X)$ consists of the elements in $\hcal^{p+q}_{\tT}(X)$ of $b$-degree $p$, and the differential $d_s$ has bidegree $(s,1-s)$.  The spectral sequence is natural in $\tT$-spaces over $\Ur$. All pages $\thetae^{*,*}_*(X)$ are $\hcal^*_{\tT}(X)$-modules, so by restriction they are modules over $\hcal^*(\mathrm U)=\HH^*(\mathrm U)\ctp \R$. The differential $d_1$ is entirely determined by the formula
$$d_1(1)=\left(\sum_{n=1}^{\infty}  \beta_nb_n\right)$$
where $\beta_n$ are pulled back from the usual primitive generators in $\HH^*(\mathrm U)$. 
\end{proposition}

The proposition is saying, in particular, that
$$\thetae^{*,*}_2(X)=\uHH_\tT^*(X).$$
We know enough about $\hcal_u^*(w_J)$ to establish the degeneration of the spectral sequence for $X=\BB_\tT(w_J)$.

\begin{proposition}
Let $w_J$ be a positive braid word. The spectral sequence $\thetae_*^{*,*}(\BB_{\tT}(w_J))$ degenerates at the second page. Hence, $\uHH_\tT^*(w_J)$ is an associated graded $\hcal^*_{\tT}(w_J)$-module of $\thetah^*_{\tT}(\BB_{\tT}(w_J))$ with respect to the $b$-degree.
\end{proposition}

\begin{proof}
From \ref{b1} in Lemma \ref{tophochdeguniv} we know that $\thetae_2^{*,*}(\BB_{\tT}(w_J))=\uHH_\tT^*(w_J)$ lives entirely in total cohomological degree $r$ modulo 2. All differentials of $\thetae_*^{*,*}(\BB_{\tT}(w_J))$ have total degree 1, so the spectral sequence must degenerate at the second page.
\end{proof}

\begin{remark} It can be verified that the filtration yielding the spectral sequence $\thetae_*^{*,*}(\BB_{\tT}(w_J))$ satisfies the Mittag-Leffler condition, guaranteeing the strong convergence of the spectral sequence. In turn, this guarantees that injectivity and surjectivity are detected on the $E_\infty$ page. Consequently, it follows that $\thetah^*_{\tT}(\BB_{\tT}(w_J))$ satisfies the relevant INS conditions that are already satisfied at the associated graded level according to \ref{c2} in Lemma \ref{tophochdeguniv}.
\end{remark}

\section{Appendix: Review of torus equivariant cohomology of flag varieties}
Recall that given a sequence of indices $J=(i_1,\dots,i_k)$, the equivariant cohomology of the Bott-Samelson space $\BTS_\tT(w_J)$ is given by the formula (\ref{deltaform}) in terms of the classes $\delta_j$. When we add such classes over all instances of a subindex $l$, we get classes
$$\hat\delta_l=\sum_{i_j=l} \delta_j$$ which appear in the formula for the right $\htor^*$-module structure on $\htor^*(\BTS_\tT(w_J))$. In fact, the class $\hat\delta_l$, as well as the bimodule structure, are pulled back from the equivariant cohomology of the flag variety $\Ur/\tT$ by means of the multiplication map 
\begin{align*}
\BTS_{\tT}(w_J)=G_{i_1}\times_T\cdots\times_T G_{i_k}/\tT & \to \Ur /\tT\\
[g_1,\dots,g_k]&\mapsto [g_1\cdots g_k].
\end{align*}

We will adapt the exposition in  \cite{kajiflag} to present their three descriptions of the structure of the cohomology $\htor^* (\Ur/\tT)$, which we have denoted by $\coebs$ in other chapters.   Let $W=N\tT/\tT$ be the Weyl group of $\Ur$, which is isomorphic to the symmetric group and is generated by the simple reflections $\sigma_1,\dots,\sigma_{r-1}$. We will need to talk about reflections associated to more general reflections: for a root $\alpha=x_i-x_j\in\htor^*=\Z[x_1,\dots,x_r]$, $i<j$, we have an associated transposition that we denote by $\sigma_\alpha=(i \,j)$.  

The terminal map $\Ur/\tT\to *$ induces a morphism $\htor^*\to \coebs$ that defines the left $\htor^*$-module structure on $\coebs$. A first description of $\coebs$ is
$$\coebs=\Z[x_1,\dots,x_r]\langle S_w\mid w\in W\rangle.$$ 
That is,  $\coebs$ is a free (left) $\Z[x_1,\dots,x_r]$-module with a basis given by the so called \emph{Schubert classes} $S_w$, indexed over the Weyl group. This is the \emph{additive} or \emph{Chevalley description} of the equivariant cohomology of $\Ur/\tT$. Given a Weyl group element $w\in W$ represented by a reduced (i.e. minimal length) word $w=\sigma_{i_1},\dots,\sigma_{i_k}$, we denote $S_w=S_{i_1,\dots,i_k}$ and the minimal length $k$ by $\ell(w)$.

Secondly, we have the \emph{GKM description} for Goresky, Kottwitz and MacPherson (see \cite{gkmequivariant}). Notice that the torus $\tT$ acts on the left of $\Ur/\tT$, and the Weyl group $W$ forms precisely the fixed point set. We can assemble the morphisms induced by inclusion of fixed points $W \hookrightarrow \Ur/\tT$ into the \emph{localization map}
\begin{equation}\label{localiz} \coebs\longrightarrow \prod_{w\in W} \htor^*.\end{equation}
The fact that this map is injective plays a central role in the equivariant setting of Schubert calculus. This injectivity is part of the content of the \emph{localization theorem}, which applies in any situation where cohomology is free over $\htor^*$ to begin with. In particular, it also applies
\begin{itemize} \item to $\Ur$ (under conjugation), where the $\tT$-fixed points form the maximal torus $\tT$;
\item to $\BTS_\tT(w_I)$, where fixed points are discrete and have the form $[\sigma_{i_1}^{\epsilon_1},\dots,\sigma_{i_k}^{\epsilon_k}]$, where each $\epsilon_i$ is either 0 or 1 and so they can be indexed by subsequences $J\in 2^I$ of $I=(i_1,\dots,i_k)$;
\item and to $\BB_\tT(w_I)$, where fixed points are discrete and have the form $[\sigma_{i_1}^{\epsilon_1},\dots,\sigma_{i_k}^{\epsilon_k}\tT]$, indexed also by subsequences $J\in 2^I$ with the condition that $w_J=\sigma_{i_1}^{\epsilon_1}\dots\sigma_{i_k}^{\epsilon_k}$ is the identity element in the symmetric group.

\end{itemize}   The result in these cases says the following:

\begin{proposition}\label{localizprop} The maps
\begin{align*}
\coebs=\htor^*(\Ur/\tT)&\longrightarrow \prod_{w\in W} \htor^*\\\coe= \htor^* (\Ur)&\longrightarrow  \htor^*(\tT) \\ \htor^* (\BTS_\tT(w_I))&\longrightarrow \prod_{J\in 2^I} \htor^* \\ \htor^* (\BB_\tT(w_I))&\longrightarrow \prod_{J\in 2^I,\, w_J=1} \htor^*(\tT)\end{align*}
induced by inclusions of fixed points are all injections.
\end{proposition}

The Schubert classes can be completely characterized through the localization map. As it happens, the tuple $(h_v)_{v\in W}$ is the image of $S_w$ if the following conditions are met:
\begin{enumerate}
\item $h_v$ is homogeneous of degree $2\ell(w)$,
\item $h_v-h_{v'}$ is a multiple of the root $\alpha$ when $v=\sigma_{\alpha}v'$ and $\ell(v')<\ell(v)$,
\item $h_v=0$ if $\ell(v)<\ell(w)$, or $\ell(v)=\ell(w)$ and $v\neq w$,
\item $h_w= \prod_{v}(-\alpha)$, with product over the  $v\in W$ such that $w=\sigma_{\alpha}v$ and $\ell(v)<\ell(w)$.
\end{enumerate}
In fact, the whole image of the localization map  is given by tuples satisfying the second condition above, known as the \emph{GKM condition}. Thinking of the equivariant cohomology of $\Ur/\tT$ as this image is the GKM description. 

Now we move to the third and final description of $\coebs$. The space $(\Ur/\tT)_{h\tT}=E\Ur\times_\tT (\Ur/\tT)$ sits in a pullback diagram
$$\begin{tikzcd}
(\Ur/\tT)_{h\tT}\ar[r]\ar[d] & E\Ur\times_{\Ur} (\Ur/\tT)\ar[d] \\
E\Ur \times_{\tT} * \ar[r] & E\Ur \times_{\Ur} *,
\end{tikzcd}$$
in which top right and bottom left corners are naturally identified with $B\tT$, while the bottom right corner is naturally identified with $B\Ur$. Collapse of the Eilenberg-Moore spectral sequence for this pullback affords the \emph{polynomial} or \emph{Borel description}
\begin{align*}
\coebs &\cong \HH^*((\Ur/\tT)_{h\tT})\\&\cong \HH^*(B\tT)\otimes_{\HH^*(B\Ur)}\HH^*(B\tT)\\& \cong\Z[x_1,\dots,x_r]\otimes_{\sym} \Z[y_1,\dots,y_r],
\end{align*}
where $\sym$ denotes symmetric polynomials in $r$ variables. Elements here are represented by polynomials $f(x_1,\dots,x_r,y_1,\dots,y_r)$. Notice that the left vertical arrow in the diagram above is precisely the homomorphism inducing the left $\htor^*$-module structure on $\coebs$. This is seen in the polynomial description by $x_i$ acting as itself on the left factor. There is also a right $\htor^*$-module induced by the top arrow, which corresponds to $x_i$ acting as $y_i$ on the right factor. 

In order to relate the polynomial description to the previous two, let us introduce a couple of pieces of notation. We use abbreviation $f(\mathbf{x},\mathbf{y})=f(x_1,\dots,x_r,y_1,\dots,y_r)$ for an element of $\coebs$ in the polynomial description. Given a symmetric group element $w$, we write $w\cdot\mathbf x=(x_{w^{-1}(1)},\dots,x_{w^{-1}(r)})$. The \emph{right divided difference operator} is
$$\Delta_i f(\mathbf x,\mathbf y)=\frac{f(\mathbf x,\mathbf y)-f(\mathbf x,\sigma_i\cdot\mathbf y)}{\alpha_i},$$
with $i=1,\dots,r-1$. For a reduced word $w=\sigma_{i_1}\dots\sigma_{i_k}$ in the symmetric group, one defines $\Delta_w=\Delta_{i_1}\circ\cdots\Delta_{i_k}$.

To go between polynomial and GKM descriptions, the crucial remark is that the $w$ component of the localization map, $\iota_w^*$, is given by the formula $$\iota_w^*f(\mathbf x,\mathbf y)=f(\mathbf x,w^{-1}\cdot\mathbf x).$$
To relate the polynomial and additive descriptions one makes use of the \emph{Newton interpolation formula}
$$f(\mathbf x,\mathbf y)=\sum_{w\in W}\Delta_w f(\mathbf x,\mathbf x) S_w.$$
This interpolation formula is useful in computations of the multiplicative structure in $\coebs$. The following proposition gives some formulas obtained this way. These formulas are relevant to the normalized $\sln$ link homology for small $n$, as it appears in Chapter \ref{boreleqcohlink}.

\begin{proposition}\label{schubertappendix} The following formulas hold in $\coebs$ for $1\leq k\leq r-1$ (by convention, any Schubert class involving a subindex less than 1 or greater than $r-1$ is zero).
  \begin{align*}
  S_{k}&=\sum_{i=1}^{k} y_i-x_i\\
S_k^2&=-\alpha_k S_k+S_{k-1,k}+S_{k+1,k} \\
S_kS_{k+1}&=S_{k+1,k}+S_{k,k+1}\\
S_k^2S_{k+1}&=-\alpha_kS_{k,k+1}-(\alpha_k+\alpha_{k+1})S_{k+1,k}\\&\qquad+S_{k-1,k,k+1}+S_{k-1,k+1,k}+S_{k+2,k+1,k}+S_{k,k+1,k}\\
S_kS_{k+1}^2&=-(\alpha_k+\alpha_{k+1})S_{k,k+1}-\alpha_{k+1}S_{k+1,k }\\&\qquad+S_{k-1,k,k+1}+S_{k,k+2,k+1}+S_{k+2,k+1,k}+S_{k,k+1,k}\end{align*}

Adding over all $k$ yields formulas:
 \begin{align*}
  \sum_{k=1}^{r-1}(x_{k+1}-x_k)S_k& = \sum_{k=1}^{r-1} S_k(S_{k}-S_{k+1})  \\
  \sum_{k=1}^{r-1}(x_{k+1}^2-x_k^2)S_k& = \sum_{k=0}^{r-1} (x_{k+1}(S_{k}-S_{k+1})+S_kS_{k+1})(S_{k}-S_{k+1})
 \end{align*}
\end{proposition}

\begin{proof}
The first five formulas follow by direct application of the Newton interpolation formula. For instance, in the second formula one sees that
$$\Delta_i S_k^2=\begin{cases} y_{k+1}-y_k+2S_k,& i=k\\0,& i\neq k \end{cases} \qquad\Delta_{\sigma_j\sigma_k} S_k^2=\Delta_{j}\Delta_{k} S_k^2=\begin{cases} 1,& j=k\pm 1\\0,& j\neq k\pm 1 \end{cases}.$$
Thus the only nonzero coefficients in Newton's formula here are $\Delta_{\sigma_k}S_k^2(\mathbf x,\mathbf x)=-\alpha_k$ and $\Delta_{\sigma_{k\pm 1}\sigma_k}S_k^2(\mathbf x,\mathbf x)=1$.

The summation formulas are then obtained from the first five. For instance, using the second and third:
\begin{align*}
\textstyle\sum_k-\alpha_kS_k&=\textstyle\sum_k S_k^2 - \textstyle\sum_k (S_{k-1,k}+S_{k+1,k}) = \textstyle\sum_k S_k^2 - \textstyle\sum_k (S_{k,k+1}+S_{k+1,k}) \\& = \textstyle\sum_k S_k^2 - \textstyle\sum_k S_kS_{k+1} = \textstyle\sum_k S_k(S_k-S_{k+1})
\end{align*}
\end{proof}

\bibliographystyle{amsplain}\bibliography{bibliography}

\end{document}